\documentclass[a4paper,11pt,twoside]{article}

\author{
\normalsize Thomas Cass \footnote{The work of Thomas Cass is supported by EPSRC Programme Grant EP/S026347/1}\\[8pt]
	\small Department of Mathematics \\
 	\small Imperial College London \\
	\small South Kensington Campus \\
	\small 180 Queen's gate \\
	\small London SW7 2AZ \\
          \\
        \small  thomas.cass@imperial.ac.uk
\and
\normalsize Gon\c calo dos Reis\footnote{G. dos Reis acknowledges support from the \emph{Funda{\c c}$\tilde{\text{a}}$o para a Ci$\hat{e}$ncia e a Tecnologia} (Portuguese Foundation for Science and Technology) through the project UID/MAT/00297/2019 (Centro de Matem\'atica e Aplica\c c$\tilde{\text{o}}$es CMA/FCT/UNL).} \\[8pt]
         \small  University of Edinburgh\\ 
         \small  School of Mathematics \\
         \small  Edinburgh, EH9 3FD, UK\\  
         \small  and \\
	\small  Centro de Matem\'atica e Aplica\c c$\tilde{\text{o}}$es \\
	\small (CMA), FCT, UNL, Portugal \\
        \small  G.dosReis@ed.ac.uk
\and
\normalsize William Salkeld \\[8pt]
         \small  University of Edinburgh\\ 
         \small  School of Mathematics \\
         \small  Edinburgh, EH9 3FD, UK\\  
         \small  \\
         \small  \\
          \\
        \small  w.j.salkeld@sms.ed.ac.uk 
}

\date{ \currenttime, \ddmmyyyydate\today}
 
\usepackage[utf8]{inputenc}
\usepackage[T1]{fontenc}
\usepackage[english]{babel}

\usepackage{charter}

\usepackage{indentfirst}
\usepackage{xcolor}
\usepackage{verbatim}

\usepackage{hyperref}

\usepackage[top=3.cm, bottom=4.0cm, left=2.2cm, right=2.2cm]{geometry}
\usepackage{amsmath}
\usepackage{amsthm}	
\usepackage{amsfonts}	
\usepackage{amssymb	
           ,bbm
           }
\usepackage{shuffle}

\usepackage{pifont}
\usepackage{enumerate}
\usepackage[nobysame
           ,alphabetic
           ,initials
           ]{amsrefs}

\usepackage{graphicx}

\usepackage{
						ulem		
} 

\numberwithin{equation}{section}
 \usepackage[nodayofweek]{datetime}

\theoremstyle{plain}
\newtheorem{theorem}{Theorem}[section]
\newtheorem{lemma}[theorem]{Lemma}
\newtheorem{proposition}[theorem]{Proposition}

\newtheorem{definition}[theorem]{Definition}

\newtheorem{remark}[theorem]{Remark}
\newtheorem{example}[theorem]{Example}
\newtheorem{assumption}[theorem]{Assumption}

\newcommand{\bE}{\mathbb{E}}

\newcommand{\bN}{\mathbb{N}}
\newcommand{\bP}{\mathbb{P}}

\newcommand{\bR}{\mathbb{R}}

\newcommand{\bW}{\mathbb{W}}
\newcommand{\bX}{\mathbb{X}}
\newcommand{\bY}{\mathbb{Y}}
\newcommand{\bZ}{\mathbb{Z}}

\newcommand{\cA}{\mathcal{A}}

\newcommand{\cC}{\mathcal{C}}
\newcommand{\cD}{\mathcal{D}}
\newcommand{\cE}{\mathcal{E}}
\newcommand{\cF}{\mathcal{F}}

\newcommand{\cH}{\mathcal{H}}

\newcommand{\cK}{\mathcal{K}}
\newcommand{\cL}{\mathcal{L}}

\newcommand{\cP}{\mathcal{P}}

\newcommand{\cR}{\mathcal{R}}
\newcommand{\cS}{\mathcal{S}}

\newcommand{\cW}{\mathcal{W}}

\newcommand{\fA}{\mathfrak{A}}
\newcommand{\fB}{\mathfrak{B}}
\newcommand{\fC}{\mathfrak{C}}

\newcommand{\fE}{\mathfrak{E}}

\newcommand{\fP}{\mathfrak{P}}
\newcommand{\fQ}{\mathfrak{Q}}

\newcommand{\fS}{\mathfrak{S}}

\newcommand{\fc}{\mathfrak{c}}

\newcommand{\fp}{\mathfrak{p}}
\newcommand{\fs}{\mathfrak{s}}
\newcommand{\fh}{\mathfrak{h}}

\newcommand{\rh}{\mathbf{h}}
\newcommand{\rw}{\mathbf{W}}
\newcommand{\rx}{\mathbf{X}}
\newcommand{\ry}{\mathbf{Y}}
\newcommand{\rId}{\mathbf{1}}
\newcommand{\rPhi}{\mathbf{\Phi}}
\newcommand{\rM}{\mathbf{M}}
\newcommand{\rN}{\mathbf{N}}
\newcommand{\rS}{\mathbf{S}}
\newcommand{\rs}{\mathbf{s}}
\newcommand{\rC}{\mathbf{C}}
\newcommand{\rc}{\mathbf{c}}
\newcommand{\rQ}{\mathbf{q}}

\DeclareMathOperator{\supp}{supp}

\DeclareMathOperator{\lip}{Lip}
\DeclareMathOperator{\diag}{Diag}

\newcommand{\dd}{\mathrm{d}}

\definecolor{darkgreen}{rgb}{0,0.35,0}

\newcommand{\1}{\mathbbm{1}}

\title{Rough functional quantization and the support of McKean-Vlasov equations}

\begin{document}

\maketitle

\renewcommand*{\thefootnote}{\arabic{footnote}}

\begin{abstract} 



We prove a representation for the support of McKean Vlasov Equations.
To do so, we construct functional quantizations for the law of Brownian motion as a measure over the (non-reflexive) Banach space of H\"older continuous paths. By solving optimal Karhunen Lo\`eve expansions and exploiting the compact embedding of Gaussian measures, we obtain a sequence of deterministic finite supported measures that converge to the law of a Brownian motion with explicit rate. We show the approximation sequence is near optimal with very favourable integrability properties and prove these approximations remain true when the paths are enhanced to rough paths. These results are of independent interest. 

The functional quantization results then yield a novel way to build deterministic, finite supported measures that approximate the law of the McKean Vlasov Equation driven by the Brownian motion which crucially avoid the use of random empirical distributions. These are then used to solve an approximate skeleton process that characterises the support of the McKean Vlasov Equation. 


We give explicit rates of convergence for the deterministic finite supported measures in rough-path H\"older metrics and  determine the size of the particle system required to accurately estimate the law of McKean Vlasov equations with respect to the H\"older norm. 




\end{abstract} 

\vspace{0.3cm}



%
%
%

\section{Introduction}

McKean Vlasov equations are Stochastic Differential Equations (SDE) with coefficients that depend on the law of the solution. This makes their analysis more involved than classical SDEs. They are sometimes referred to as mean-field or distribution dependent SDEs and were first studied in \cite{McKean1966}. These equations describe a limiting behaviour of individual particles having diffusive dynamics and which interact with each other in a ``mean-field'' sense. Hence, the motion of a single particle is determined in terms of the motion of all other particles. The solutions of these mean field systems of equations are a powerful tool in understanding statistical mechanics such as Boltzmann Equations. Applications are numerous and vary from opinion dynamics \cite{hegselmann2002opinion}, the dynamics of granular materials \cites{benedetto1998non, bolley2013uniform, cattiaux2008probabilistic}, molecular and fluid dynamics \cite{pope2001turbulent}, interacting agents in economics or social networks \cite{carmona2013control}, mathematical biology \cites{keller1971model, burger2007aggregation}, Galactic dynamics \cite{binney2011galactic}, droplet growth\cite{conlon2017non}, Plasma Physics \cite{bittencourt2013fundamentals}, interacting neurons \cite{delarue2015global} and deep learning neural networks \cite{hu2019mean}. See also \cites{CarmonaDelarue2017book1, CarmonaDelarue2017book2} and references therein for a detailed exploration of the applications of McKean Vlasov Equations. 

McKean Vlasov Equations have also been studied in the context of rough paths. In the first work \cite{CassLyonsEvolving}, the authors treat the measure dependency in the drift term as a bounded variation Banach valued operator. Thus the interactive forces can be calculated using Banach valued Young integrals and there is no need to exploit the rough path structures beyond what is already necessary to incorporate the noise. Their approach is limited only by the assumption of no measure dependencies in the diffusion term.  Later, in \cite{2018arXiv180205882B} the authors develop the new framework of \emph{Probabilistic Rough Paths}. This insightful development encodes the law of the noise into the rough path, allowing the noise to interact with the measure dependencies and opening up the collection of possible diffusion terms to include adequately regular measure dependencies. Other works that study McKean Vlasov Equations via rough paths include \cite{deuschel2017enhanced}, \cite{coghi2018pathwise} and \cite{coghi2019rough}. 

The support of a measure is the smallest closed set of full measure. Thus the \emph{Support theorem} for the law of an SDE characterises the set of admissible paths that the SDE can take with respect to a particular choice of topology. The first work studying the support of an SDE was \cite{stroock1972support} where the law of an SDE is characterised in terms of the supremum norm and the authors goal was to establish a Strong Maximum principle for a class of Elliptic Partial Differential Equations. This was later extended to a wide class of processes in \cite{gyongy1990approximation}. Later, a support theorem with respect to the H\"older norm was established in \cite{ben1994holder}, and for a much wider class of norms in \cite{gyongy1995approximation}. These works laid the groundwork for the later publication \cite{ledoux2002large} which studies the support of the solution law of a Rough Differential Equation driven by a Gaussian white noise. In \cite{friz2006levy} it is shown that the continuity of the It\^o-Lyons map means that the proof of a support theorem can be reduced to establishing a characterisation of the support of the driving noise in an adequately rich topology. 

Support theorem results have been key in some other applications, for example, a support theorem for SDEs with jump noise was crucial in showing \emph{Exponential Ergodicity} in \cite{Kulik2009}. One of the conditions the authors require is \emph{Topological Irreducibility}, that for any two points, there is a path of the jump process that passes between them in finite time. This can be verified by finding an expression for the support of the law. Support theorems are also central in the establishment of \emph{Stochastic Invariance principle}. A stochastic process is said to be invariant of a closed set $\cD\subseteq \bR^d$ if the solution starts and remains on the set $\cD$ $\bP$-almost surely $\forall t\in[0,T]$. This problem was first studied in \cite{aubin1990stochastic}. More recently, Stochastic Invariance has been studied in \cite{zabczyk2000stochastic}, \cite{buckdahn2010another} and \cite{FilipovicTappeTeichmann2014}. In general, support theorems continue to draw attention from a wide range of academics, see  \cite{chouk2018support}, \cite{cont2018support} and \cite{2019arXiv190905526H}. Lastly, a motivation to study support theorem results for McKean Vlasov Equations is the recent link between this class of equations, deep learning (or \emph{rich learning}) and ergodicity, see \cite{hu2019mean}. 
\bigskip

The following useful method for proving a classical support theorem can be found in \cite{millet1994simple}.
\begin{theorem}
Let $(\Omega, \cF, \bP)$ be a probability space containing a Brownian motion and let $E$ be a separable Banach space. Let $\cH$ be the reproducing kernel Hilbert space of Brownian motion. Let $X:\Omega \to E$ be a random variable and let $\Phi:\cH \to E$ be a measurable map. 
\begin{enumerate}
\item Suppose there exists a sequence of random variables $H_n:\Omega \to \cH$ such that for any $\varepsilon>0$
\begin{equation}
\label{eq:ClassicalSupportThm1}
\lim_{n\to \infty} \bP\Big[ \| X(\cdot) - \Phi(H_n(\cdot))\|_{E} > \varepsilon \Big] = 0. 
\end{equation}
Then $\supp(\cL^X) \subset \overline{\Phi(\cH)}^E$. 
\item Suppose there exists a sequence of measure transforms $T_n^h$ such that $\bP\circ T_n^h$ is absolutely continuous with respect to $\bP$ and for any $\varepsilon >0$
\begin{equation}
\label{eq:ClassicalSupportThm2}
\limsup_{n\to \infty} \bP\Big[ \| X(T_n^h(\cdot)) - \Phi(h)\|_E < \varepsilon \Big] >0. 
\end{equation}
Then $\overline{\Phi(\cH)}^E \subset \supp (\cL^X)$. 
\end{enumerate}
If both \eqref{eq:ClassicalSupportThm1} and \eqref{eq:ClassicalSupportThm2} are satisfied, then $\overline{\Phi(\cH)}^E = \supp (\cL^X)$ and $\Phi$ is called the \emph{Skeleton Process} of the random variable $X$, see \cite{caballero1997composition}.  
\end{theorem}
Equation \eqref{eq:ClassicalSupportThm1} is sometimes referred to as the Wong Zakai implication due to its similarity with the Wong Zakai theorem. Equation \eqref{eq:ClassicalSupportThm2} is sometimes referred to as the Cameron Martin implication because the proof involves exploiting the absolute continuity of Cameron Martin transforms on Wiener space.

\subsubsection*{Our contribution}
Proving a support theorem for McKean Vlasov Equations is more challenging than verifying Equations \eqref{eq:ClassicalSupportThm1} and \eqref{eq:ClassicalSupportThm2}. The knowledgeable reader will realise that for McKean Vlasov Equations, the Skeleton process is itself dependent on the law of the solution of the McKean Vlasov Equation so the law must be known exogenously in order to solve any Skeleton process path. This is in contrast to the Skeleton process used in \cite{dos2019freidlin} where the measure dependency is replaced by a Dirac following the skeleton process driven by a constant $0$ noise. 

Before tackling the methods to represent the support of McKean Vlasov equations we address, separately and of independent interest, the \emph{Quantization problem for the law of a Brownian motion as a measure over the collection of H\"older continuous rough paths}. The quantization problem for Gaussian measures for Hilbert spaces was first studied in \cite{LUSCHGY2002486}, but for Banach spaces, the problem is more challenging with the optimal rate of convergence solved in \cite{graf2003functional} and separately \cite{dereich2003link}. These methods rely on the small ball probabilities of Brownian motion, see \cite{baldi1992some}, a tool to measure the compactness of the reproducing kernel Hilbert space unit ball contained in the Banach space. 

Using functional analytic methods, we construct a quantization for the law of the Brownian motion that has a rate of convergence that is asymptotically equivalent to the optimal rate of convergence. Our quantization is not optimal, indeed such a quantization does not exist due to the non-weak compactness of the H\"older unit ball. We choose to sacrifice optimality in order to retain certain key properties that allow us to estimate the law of the quantization accurately. To do this, we construct a Karhunen Lo\`eve expansion that optimally approximates the Brownian motion with respect to the H\"older norm. Although this representation for Brownian motion is well documented \cite{herrmann2013stochastic}, it is not so well known that the wavelet representation comes from the spectral decomposition of the covariance kernel and so it embodies the optimal approximation by a finite dimensional Gaussian. These quantizations are then enhanced to rough paths and we prove that the rate of convergence for the E
enhanced quantization to the enhanced Brownian motion is asymptotically the same. 

Quantization for rough paths has been first studied in \cite{pages2011convergence}. The choice of Karhunen Lo\`eve expansion and method of construction used in this work, namely the trigonometric functions, best suits approximations of Brownian motion in the $L^2([0,T];\bR^{d'})$ norm. Although this is enough to ensure convergence in the H\"older norm, it is far from efficient and (to the best of our knowledge) no literature exists for rates of convergence. Our approach is demonstrated to be arbitrarily close to optimal and we provide upper and lower bounds on the rate of convergence. 

The key advantage of this deterministic construction over the use of empirical distribution used in McKean Vlasov numerics is that we avoid all difficulties with characterising the support (a deterministic set) from random samples. For instance, the almost-sure rate of convergence for an Empirical distribution may, for a particular sample, be too poor to be of any effective use. 

By solving the system of interacting Rough Differential Equations driven by a H\"older quantization of the Brownian motion and exploiting the continuity properties of Rough Differential Equations, we obtain a \emph{deterministic, finite support measure} that approximates the law of the McKean Vlasov Equation without having to solve the law explicitly. One could equivalently obtain the solution law by solving the non-linear Fokker Planck equation, but a novelty of this work is to attain the law without having to resort to PDE methods. We initiate our study by developing our results entirely within the framework established in \cite{CassLyonsEvolving}. 

It is also worth emphasising that the rate of convergence that we obtain in Theorem \ref{thm:ActualRateCon4QuantizationRough} is, at face value, much slower than other well known  methods for sampling a measure. The reason for this is we approximate in pathspace rather than for any fixed choice of time. Thus our quantization encodes both information about the path of a Brownian motion and the H\"older regularity. 
\medskip

Finally, to \emph{prove the support of McKean Vlasov Equation} we develop a novel method by considering the sequences of pairs $(H_n, \cL_n)_{n\in \bN}$ and $(T^h_n, \cL_n)_{n\in \bN}$ where $(\cL_n)_{n\in \bN}$ is a sequence of measures that converge to the law of the McKean Vlasov Equation. However, for each $n\in \bN$, the Skeleton process $\Phi(h, \cL_n)$ driven by $\cL_n$ and a reproducing kernel Hilbert space path $h$ are not necessarily contained in the support even though they are a good approximation of a path that is contained in the support. Thus our statement for the support takes the form (see Theorem \ref{thm:SupportTheorem1} below)
$$
\supp(\cL) = \bigcap_{n\in \bN} \overline{\bigcup_{m\geq n} \Big\{ \Phi(h, \cL_m): h\in \cH \Big\} }^{\alpha-\textrm{H\"ol}}.
$$

In this paper, we prove two support theorems, see Theorem \ref{thm:SupportTheorem1} and Theorem \ref{thm:SupportTheorem1RandInit}. The first is for McKean Vlasov Equations where the initial condition is deterministic while the second is a extension of this result for McKean Vlasov Equations with random initial condition. The proof of the extension is simple and follows from \cite{caballero1997composition} so we focus predominantly on the first case. 

Lastly, we highlight some similarities between our statement (and proof) of the support theorem and the brilliant approach by Hairer and Sch\"onbauer in \cite{2019arXiv190905526H}*{Theorem 1.3} studying the support of the solution to Singular Stochastic Partial Differential Equations using Regularity Structures. Both results are stated in this non-standard way (crucially due to the approach), namely, as the restriction to the limit points of a collection of smooth paths. Our work was developed independently and was first presented at the 10th Oxford-Berlin Young Researchers Meeting on Applied Stochastic Analysis in December 2018. 
\bigskip

We point out that we make full use of rough paths techniques in our manuscript but restrict ourselves to McKean Vlasov equations driven by Brownian motion (as opposed to general Gaussian noises as in \cite{CassLyonsEvolving}). The reason for this is the challenges associated with constructing a \emph{Truncation} and \emph{Quantization} for a general Gaussian driving noise. In particular, for general Gaussian processes one loses the neat truncation properties given by the reproducing kernel Hilbert space being spanned by orthonormal Schauder functions. It is noteworthy to point out that there is no difference between the construction of the quantization of a Brownian motion and that of a Brownian bridge. This is because Schauder wavelets are also orthonormal in the reproducing kernel Hilbert space of a Brownian Bridge. 

All in all, there are several works that address quantization of more general Gaussian processes, for instance \cite{luschgy2006functional}, \cite{dereich2006high} and \cite{luschgy2008functional}. These works study functional quantization with respect to the $L^p$ norm and supremum norm over time rather than the H\"older norm and so do not properly encode all of the necessary regularity information to solve Rough Differential Equations efficiently. We will shortly address this problem.


\bigskip
This work is organized as follows. We recall several crucial definition and results in Section \ref{sec:preliminaries}. In Section \ref{sec:QuantizationBM} we discuss the construction of a finite support measure that approximates the law of an enhanced Brownian motion as a measure over the space of geometric rough paths using functional quantization. In Section \ref{sec:MVSDEsrough} we visit the construction of rough McKean Vlasov Equations. The support theorem for the class of McKean Vlasov equations addressed in this work in presented in Section \ref{sec:SupportTheo}.


\section{Preliminaries}
\label{sec:preliminaries}

\subsection{Notation and spaces}

We denote by $\bN=\{1,2,\cdots\}$ the set of natural numbers and $\bN_0=\bN\cup \{0\}$, $\bZ$ and $\bR$ denote the set of integers and real numbers respectively. $\bR^+=[0,\infty)$. By $\lfloor x \rfloor$ we denote the largest integer less than or equal to $x\in \bR$. $\1_A$ denotes the usual indicator function over some set $A$. Let $e_j$ be the unit vector in the $j^{th}$ component. 

For sequences $(f_n)_{n\in \bN}$ and $(g_n)_{n\in\bN}$, we denote 
\begin{align*}
f_n \lesssim g_n \ \ \iff  \ \ \limsup_{n\to \infty} \frac{f_n}{g_n}\leq C, 
\qquad \textrm{and}\qquad 
f_n \gtrsim g_n \ \ \iff  \ \ \liminf_{n\to \infty} \frac{f_n}{g_n}\geq C. 
\end{align*}
where $C$ is a positive constant independent of the limiting variable. When $f_n \lesssim g_n$ and $f_n \gtrsim g_n$, we say $f_n \approx g_n$. This is distinct from
\begin{align*}
f_n \sim g_n \ \ \iff \ \ \lim_{n\to \infty} \frac{f_n}{g_n} = 1. 
\end{align*}

Let $C^0([0,T]; \bR^d)$ be the space of continuous functions over the interval $[0,T]$ taking values in the vector space $\bR^d$ that start at $0$ paired with the supremum norm. For $\alpha\in(0,1)$, we define the $\alpha$-H\"older norm
$$
\| \psi\|_\alpha = \sup_{s, t\in[0,T]} \frac{ | \psi(t) - \psi(s)|}{|t-s|^\alpha}. 
$$

Let $C^\alpha([0,T]; \bR^d)$ be the subset of $C^0([0,T]; \bR^d)$ such that $\|\cdot\|_\alpha$ is finite. For $\alpha<\beta<1$, $\beta$-H\"older continuous paths are compactly embedded in the space of $\alpha$-H\"older continuous paths e.g. the spaces $C^{\beta}([0,T]; \bR^{d'}) \Subset C^{\alpha}([0,T]; \bR^{d'})$. Although the space $C^{\alpha}([0,T]; \bR^{d'})$ is not separable, the subset $C^{\alpha, 0}([0,T]; \bR^{d'}):= \overline{C^{\beta}([0,T]; \bR^{d'})}^{\alpha-\mbox{\tiny{H\"older}}}$ is separable. 

Let $(\Omega, \cF, \bP)$ be a probability space carrying a $d'$-dimensional Brownian Motion on the interval $[0,T]$ where throughout $T>0$. The Filtration on this space satisfies the usual assumptions. We denote by $\bE$ and $\bE[\cdot|\cF_t]$ the usual expectation and conditional expectation operator (with respect to $\bP$) respectively. For a random variable $X$ we denote its probability distribution (or Law) by $\cL^X$; the law of a process $(Y_t)_{t\in[0,T]}$  at time $t$ is denoted by $\cL_t^Y$.

For $\mu$, a probability measure on $(E, \cE)$, we define the support of $\mu$, denoted $\supp(\mu)$, to be the set of points $x\in E$ such that every open neighbourhood of $x$ has positive measure. Equivalently, it is the smallest closed set of full measure.
%

\subsection{Gaussian Theory}

We briefly summarise some standard results relating to Gaussian processes and Gaussian measures. 
\begin{definition}
\label{definition:GaussianMeasure1}
A centred Gaussian measure $\cL$ on a real separable Banach space $E$ equipped with its Borel $\sigma$-algebra $\cE$ is a Borel probability measure on $(E, \cE)$ such that the law of each continuous linear functional on $E$ is Gaussian with mean 0. 
\end{definition}

Let $E$ be a separable Banach space. Then it is well known that the Borel $\sigma$-algebra and the cylindrical $\sigma$-algebra are the same (see for example \cite{bogachev1998gaussian}). Let $\cH$ be the Reproducing Kernel Hilbert Space (RKHS) of the Gaussian measure. We denote the unit ball in the RKHS norm as $\cK$. It is well known that the set $\cK$ is compact in the Banach space topology of $E$ and $\cH$ is dense in the support of $\cL$. 

We consider the law of a Gaussian process as a measure on pathspace, that is a measure over the space of continuous paths starting at $0\in \bR^{d'}$. We are interested in the space of $\alpha$-H\"older continuous paths for $\alpha<\tfrac{1}{2\varrho}$ and the topology induced by this norm where $\varrho\in[1, 2)$. For any choice of $\alpha<\tfrac{1}{2\varrho}$, we can find $\alpha<\alpha'<\tfrac{1}{2\varrho}$ for which the Gaussian process will be $\alpha'$-H\"older continuous. Therefore, we will always have that the Gaussian process takes values in $C^{\alpha, 0}([0,T]; \bR^{d'})$ and we do not concern ourselves with separability further. 

\begin{definition}[Haar Functions]
\label{definition:HaarSchauder}
Let $t\in[0,T]$. For $p\in \bN_0$ and $m\in\{1, ..., 2^p\}$, define the sequence of values $t_{pm}^0 = \tfrac{(m-1)T}{2^p}$, $t_{pm}^1 = \tfrac{(2m-1)T}{2^{p+1}}$ and $t_{pm}^2 = \tfrac{mT}{2^p}$. Define the functions $H_{00}(t)=1$ and
\begin{equation*}
H_{pm}(t)=\begin{cases}
\sqrt{\tfrac{2^p}{T}}, & \text{if $t \in [t_{pm}^0, t_{pm}^1)$},\\
-\sqrt{\tfrac{2^p}{T}}, & \text{if $t \in [t_{pm}^1, t_{pm}^2)$},\\
0, & \text{otherwise}.
\end{cases}
\end{equation*}
These are called the Haar functions, a orthonormal collection of functions in $L^2([0,T]; \bR)$.

The Schauder function are similarly defined $G_{pm}(t) = \int_0^t H_{pm}(s) ds$. 
\end{definition}

The Haar functions form an orthonormal basis on the space $L^2([0,T]; \bR)$ with the canonical inner product. Therefore, we define the Fourier coefficients $\psi_{pm} = \int_0^T H_{pm}(s) \psi(s)ds$ and the set
$$
\Lambda:=\Big\{ (p, m): p\in \bN_0, m\in\{ 1, ..., 2^p\}\Big\}\cup \Big\{ (0, 0)\Big\}. 
$$ 
We do not include the pair $(p, m) = (-1, 0)$ as throughout we will be dealing with Gaussian processes which are 0 at $t=0$. 

Next, for some continuous path $\psi$ taking values in $\bR^{d'}$, we define the Schauder Fourier coefficients to be
\begin{align}
\label{eq:fpmCoefficients}
	\psi_{pm} 
	:= \langle {H_{pm},\dd \psi}\rangle 
	:= \sqrt{\tfrac{2^p}{T}}\left[2\psi(t^1_{pm})-\psi(t^0_{pm})-\psi(t^2_{pm})\right]\in \bR^{d'}, 
	\quad \text{for}\quad 
	(p,m)\in \Lambda;
\end{align}
additionally $\psi_{00} := \langle {H_{00},\dd \psi}\rangle = \psi(1)-\psi(0)$. Let us denote $\Lambda_N=\{(p, m)\in \Lambda : p\leq N\}$ as a truncation of $\Lambda$. 

The following Theorem, often referred to as the Cielsielski Isomorphism, provides the link between wavelet theory and rough paths.
 
\begin{theorem}[\cite{herrmann2013stochastic}]
\label{thm:Ciecielski}
For $\alpha>0$, let $\|\cdot \|_{\alpha}$ be the $\alpha$-H\"older norm. Let $\psi \in C^0([0,T]; \bR^{d'})$. We have that $\|\cdot\|_\alpha$ is equivalent to 
\begin{equation}
\label{eq:CielsielskiNorm1}
\|\psi\|_\alpha' =\sup_{(p,m)\in \Lambda} 2^{(\alpha-1/2)p} |\psi_{pm}|.
\end{equation}
If, in addition, we have that
$$
\lim_{p\to\infty} 2^{p(\alpha-1/2)} \sup_{1\leq m\leq 2^p } |\psi_{pm}| = 0
$$ 
we say that $\psi \in C^{\alpha, 0}([0,T]; \bR^{d'})$. This space is a separable subset of $C^\alpha([0,T]; \bR^{d'})$. 
\end{theorem}

\begin{example}[Cielsielski Representation of Brownian motion]
Due to the orthogonality of the Schauder functions in the RKHS of Brownian motion, we can represent Brownian motion as
\begin{equation}
\label{eq:CielsielskiRepresentation}
W_t = \sum_{(p,m)\in \Lambda} W_{pm} G_{pm}(t) \quad t\in [0,T]
\end{equation}
where $W_{pm}$ is a sequence of $d'$-dimensional, independent, standard normally distributed random variables. Thus
$$
\| W\|_\alpha = \sup_{(p,m) \in \Lambda} 2^{p(\alpha-1/2)} |W_{pm}|. 
$$
\end{example}

\subsection{Measures and Approximation}

For $E$ a complete, separable metric space with Borel $\sigma$-algebra $\cE$, let $\cP_r(E)$ be the set of all Borel measures over $(E, \cE)$ which have finite $r^{th}$ moments. 

\begin{definition}
\label{defn:WassersteinDistance}
Let $\mu, \nu\in \cP_r(E)$. We define the Wasserstein $r$-distance $\bW^{(r)}_{E, d}:\cP_r(E)\times \cP_r(E) \to \bR^+$ to be
\begin{equation}
\label{eq:defWasserstein}
\bW^{(r)}_{E, d}(\mu, \nu) = \left( \inf_{\gamma \in \cP(E\times E)} \int_{E\times E} d(x, y)^r \gamma(dx, dy) \right)^{\tfrac{1}{r}}
\end{equation}
where $\gamma$ is a joint distribution over $E\times E$ which has marginals $\mu$ and $\nu$. When the space the measure is defined on is clear, we write $\bW^{(r)}_{d}$ where $d$ is the metric over $E$. 
\end{definition}

The problem of finding a measure $\gamma\in \cP_2(E\times E)$ that minimises \eqref{eq:defWasserstein} is sometimes referred to as the Kantorovich problem and $\gamma$ is called the transport plan of $\mu$ and $\nu$. The choice of $r=2$ is common throughout literature. However, we will also be interested in the case $r=1$. The $r$-Wasserstein distance induces the topology of weak convergence of measure as well as convergence in moments of order up to and including $r$. The Wasserstein distance is a metric, but the metric does not induce a norm. The Wasserstein distance is homogeneous but not translation invariant. The space $\cP_2(E)$ is complete and separable with respect to the Wasserstein metric (see \cite{Bolley2008}).

\subsubsection{Quantization of Measures}

We provide a brief introduction to the field of quantization. For further details, see \cite{graf2007foundations}.

\begin{definition}
Let $\cL$ be a measure on a separable Banach space $E$ endowed with the Borel $\sigma$-algebra such that $\cL\in \cP_2(E)$ and all sets of codimension $1$ have null $\cL$-measure. 

Let $I$ be a countable index, let $\fS:=\{ \fs_i, i\in I\}$ be a partition of $E$ and let $\fC:=\{\fc_i\in E, i\in I\}$ be a codebook. For any partition $\fS$ and codebook $\fC$, we define a quantization $q:E \to E$ by
\begin{align*}
q(x) = \fc_i &\quad \mbox{for $x\in \fs_i$}, \qquad q(E) = \fC
\end{align*}
so that
$$
\cL \circ q^{-1}(\cdot) = \sum_{i\in I} \cL(\fs_i) \delta_{\fc_i} (\cdot) \in \cP_2(E). 
$$
The collection of all quantizations is denoted $\fQ$. 
\end{definition}

\begin{definition}
\label{dfn:ComplexityMapping}
Let $\fP \subset [0,1]^{\bN}$ be the set of probability vectors e.g. for every $\fp = (\fp_i)_{i\in \bN}$, we have $\fp_i\in[0,1]$ and $\sum_{i\in \bN} \fp_i = 1$. 
\end{definition}

Given a partition $\fS$ of a space $E$, we have that the sequence $(\cL(\fs_i))_{\fs_i\in \fS}$ is a probability vector. 

\begin{definition}[Optimal Quantizers]
\label{dfn:OptimalQuantizer}
Let $n\in \bN$ and $r\in[1, \infty)$. The minimal $n^{th}$ quantization error of order $r$ of a measure $\cL$ on a separable Banach space $E$ is defined to be
$$
\fE_{n, r}(\cL) = \inf\Bigg\{ \Big( \int_E \min_{\fc\in \fC} \| x-\fc \|_E^r d\cL(x) \Big)^{\tfrac{1}{r}}: \fC\subset E, 1\leq |\fC| \leq n\Bigg\}. 
$$

A Codebook $\fC = \{\fc_i, i\in I\}$ with $1\leq|\fC|\leq n$ is called an $n$-optimal set of centres of $\cL$ (of order r) if
$$
\fE_{n, r}(\cL) = \Big( \int_E \min_{i=1, ..., n} \| x-\fc_i\|_E^r d\cL(x) \Big)^{\tfrac{1}{r}}
$$
\end{definition}

Given a finite collection of elements $(\fc_i)_{i=1, ..., n}$, the optimal way to choose the partition of $E$ is to use the nearest neighbour rule which corresponds to the Voronoi partition
\begin{equation}
\label{eq:VoronoiSetDef}
\fs\Big(\fc_i\Big| (\fc_j)_{j=1, ..., n} \Big):= \Big\{ x \in E: \| x - \fc_i\| = \min_{j=1, ..., n} \|x - \fc_j\| \Big\}
\end{equation}
provided the boundary of the Voronoi sets has measure 0. Sets of the form \eqref{eq:VoronoiSetDef} are called \emph{Voronoi sets}. Similarly, given a finite partition $( \fs_i)_{i=1, ..., n}$ of $E$, the optimal choice of codebook is the centres of mass for the sets $\fs_i$ with respect to the measure $\cL$. For brevity of notation, we write $\fE_n:=\fE_{n, 2}$. 

\subsubsection{Stationary Quantization}

A Stationary set is a codebook with a special property: the Voronoi sets generated by codebook have barycentres equal to the codebook. 

\begin{definition}
Let $E$ be a separable Banach space with Borel $\sigma$-algebra $\cE$, let $n\in \bN$ and let $\cL$ be a measure on $(E, \cE)$ such that and all subsets of codimension $1$ have null $\cL$-measure. Let $\fC\subset E$ satisfy $|\fC| = n$

Suppose that the Voronoi partition $\fS$ of $E$ generated by the elements of $\fC$, containing the collection of sets $\fs_i: = \{y\in E: \min_{j=1, ..., n}\|y-\fc_j\| = \|y-\fc_i\| \}$ satisfies that
$$
\frac{1}{\cL(\fs_i)} \int_{\fs_i} y d\cL(y) =\fc_i. 
$$
Then we call the codebook $\fC$ an \emph{$n$-stationary set} of the law $\cL$. 
\end{definition}

\begin{theorem}[\cite{laloe20101}*{Theorem 2.1}]
\label{theorem:ExistenceStationaryQuant}
Let $E$ be a reflexive, separable Banach space and let $\cL$ be a measure on $(E, \cE)$. For $\fc_i\in E$, define $\fA:E^n \to \bR$ by
$$
\fA(\fc_1, ..., \fc_n) = \int_E \min_{i=1, ..., n} \|y - \fc_i\|_E^2 d\cL(y)
$$
e.g. $\fA(\fc_1, ..., \fc_n)$ is the mean square error between the measure $\cL$ and the quantization with codebook $\{\fc_1, ..., \fc_n\}$ and partition equal to the Voronoi sets of the codebook. 

Then $\fA$ admits at least one minimum, and so an n-stationary set exists. 
\end{theorem}

\begin{remark}
\label{rem:ReflexivityImportance}
The proof of the above result relies on the Assumption that the Banach space $E$ is reflexive. In particular, for a non-reflexive space the unit ball will be weak-$*$ compact but not weak compact (see \cite{fabian2013functional}*{Theorem 3.31}). By contrast, the functional $\fA$ can be shown to be weak lower semicontinuous but the proof does not extend to weak-$*$ lower semicontinuity. 

In particular, we are interested in Gaussian measures over the Banach space $C^{\alpha, 0}([0,T]; \bR^{d'})$, which is not reflexive and so Theorem \ref{theorem:ExistenceStationaryQuant} does not apply. 

Lastly, it is not clear whether a stationary quantization exists in general. 
\end{remark}

\begin{lemma}
\label{lem:StationaryRKHS}
Let $\cL$ be a centred Gaussian measure taking values on the Banach space $E$ and suppose that an $n$-stationary set exists. Let $\fC$ be an $n$-stationary set. Then $\fC \subset \cH$. 
\end{lemma}

\begin{proof}
This proof is based on a similar argument first presented in \cite{LUSCHGY2002486} which focuses solely on Hilbert spaces. Using that the $n$-stationary set exists, we have that for any $\fc \in \cC$
$$
\fc = \int_\fs x d\cL(x) = \int_E x \cdot \frac{\1_{\fs}(x)}{\cL(\fs)} d\cL(x)
$$
Next, we use that $\tfrac{\1_{\fs}}{\cL(\fs)}$ is a square integrable function with respect to $\cL$ on $E$ and use Definition of the RKHS to conclude that the right hand side of this equation must be an element of $\cH$. Therefore $\fc\in \cH$
\end{proof}

\begin{remark}
\label{lemma:convexcheat}
In particular, if $q_n(W)$ denotes the quantized random variable $W$, then the Stationary quantization has the property that
$$
q_n(W) = \bE[ W| \cF_n], 
$$
where $\cF_n$ is the $\sigma$-algebra generated by the partition of $q_n$. This is a particularly useful property when it comes to establishing uniform integrability of quantizations due to the following simple argument: 

Let $\phi$ be a convex function on a Banach space $E$. Then
\begin{align}
\label{eq:StationaryQuantConvex}
\sup_{n\in \bN} \bE\Big[ \phi(q_n(W))\Big] =\sup_{n\in \bN} \bE\Big[ \phi\big( \bE[W|\cF_n]\big)\Big] 
\leq \sup_{n\in \bN} \bE\Big[ \bE\big[\phi (W) |\cF_n\big] \Big] = \bE\Big[ \phi (W)\Big]. 
\end{align}
\end{remark}

\begin{lemma}
\label{lem:DivisionQuantization+Truncation}
Let $\cL$ be a non-degenerate Gaussian measure over $E$ with RKHS $\cH$. Let $U$ be a finite dimensional subspace of $\cH$ and let $P_U$ be the orthogonal projection operator from $\cH$ to $U$ extended to $E=\overline{\cH}^E$. Then $\forall r>1$
$$
\fE_{n,r}(\cL) \lesssim \fE_{n,r}\big(\cL\circ (P_U)^{-1}\big) + \Bigg( \int_{E} \big\| x - P_U[x]\big\|_E^r d\cL(x) \Bigg)^{1/r}. 
$$
\end{lemma}

In particular, when the measure $\cL$ is in some sense ``concentrated" on a finite dimensional linear subspace of the Banach space $E$, then the quantization problem can be simplified to a finite dimensional problem. 

\begin{proof}
See Appendix \ref{sec:AppendixA}
\end{proof}

\subsubsection{Rate of Convergence for Quantization}
\label{subsubsec:SmallBall}

In the finite dimensional setting, the minimal quantization error is well understood (see \cite{graf2007foundations}). Let $\cL$ be a measure over a $d$-dimensional vector space. Then 
\begin{equation}
\label{pro:GrafLuschgyQuantizationRate}
\fE_{n, r}(\cL) \approx n^{1/d}.
\end{equation}
However, for a Gaussian measure over a Banach space $E$, the limit $d\to \infty$ is no longer meaningful. In both \cite{dereich2003link} and \cite{graf2003functional}, the authors investigate the relation between the minimal quantization error and the probabilities of small balls. 

\begin{theorem}[\cite{dereich2003link}, \cite{graf2003functional}]
\label{cor:AsymtoticRateConQuantErr}
Let $\cL^W$ be a Gaussian measure over a Banach space $E$. Let $\fB_W$ be the small ball probability of $\cL^W$ defined by $\fB_W(\varepsilon):=-\log \cL\big[ \{ x\in E: \|x\|_E<\varepsilon\}\big]$. Then for any choice of $r\geq1$
$$
\fE_{n, r}(\cL^W) \approx (\fB_W)^{-1} \Big( \log(n)\Big)
$$
as $n\to \infty$. In particular, let $\cL^W$ be the law of Brownian motion over $C^{\alpha, 0}([0,T]; \bR^{d'})$. Then by the results of \cite{baldi1992some}
\begin{equation}
\label{eq:AsymtoticRateConQuantErr}
\fE_{n, r}(\cL^W) \approx d' \Big( \log(n^{1/d'}) \Big)^{\alpha-1/2}. 
\end{equation}

\end{theorem}

In particular, Equation \eqref{eq:AsymtoticRateConQuantErr} provides us with a lower bound that the error of the quantization for Brownian motion cannot outstrip. However, as already explained in Remark \ref{rem:ReflexivityImportance}, there may not exist a stationary quantization that attains $\fE_{n, r}(\cL^W)$. 

A remarkable aspect of \cite{dereich2003link} is that the authors additionally prove that the mean square error between an empirical measure and the true Gaussian measure in the Wasserstein distance converges at the same rate as the optimal quantization error. 

\subsection{Rough Paths}
\label{subsec:RoughPaths}

Throughout this paper, we will use the notation for increments of a path $X_{s, t} = X_t - X_s$ for $s\leq t$. Rough paths were first introduced in \cite{lyons1998differential}. For a detailed overview of rough path theory, see \cite{friz2010multidimensional}, \cite{frizhairer2014} and \cite{lyons2002system}. For a self-contained exposition, the reader can find a primer on rough paths in Appendix \ref{Appendix:Primer}. 

\subsubsection{The lift of Gaussian Processes}

%

In \cite{friz2010differential}, the authors prove that when the covariance operator of the Gaussian satisfies a $p$-variation condition, the path of the Gaussian can be lifted to a finite $p$-variation or $\alpha$-H\"older continuous rough path. 

\begin{assumption}
\label{assumption:GaussianRegularity}
Let $\cL^W$ be the law of a continuous, centred Gaussian process with independent components taking values in $\bR^{d'}$ and covariance operator $\cR$ such that $\exists \varrho\in[1, 2)$ and $M<\infty$ with
$$
\| \cR\|_{\varrho; [s, t]^2} \leq M |t-s|^{1/ \varrho}. 
$$
\end{assumption}


\section{Approximation of Brownian motion}
\label{sec:QuantizationBM}

The goal of this Section is to construct a finite support measure that approximates the law of an enhanced Brownian motion as a measure over the space of geometric rough paths. 

\subsection{Truncation of Brownian Motion}
\label{subsec:TruncationofBrownianMotion}

Using the Cielsielski representation for Brownian motion from Equation \eqref{eq:CielsielskiRepresentation}, we can obtain a finite dimensional Gaussian measure on $C^{\alpha, 0}([0,T]; \bR^{d'})$ 
\begin{equation}
\label{eq:dfn:TruncatedBrownianMotion}
W^{N}_t = \sum_{(p, m)\in \Lambda_N} W_{pm} G_{pm}(t),
\end{equation}
which approximates Brownian motion. Let us briefly describe some of the properties this random variable:
\begin{itemize}
\item $W^{N}$ is a Gaussian measure on $C^{\alpha, 0}([0,T]; \bR^{d'})$ with RKHS $\cH_N=\big( span_{(p,m)\in \Lambda_N}\{ G_{pm}\}\big)^{\times d'}$. 
\item As a finite dimensional Gaussian, the support of $W^{N}$ is just $\cH_N$. This is equal to the space of $d'$-dimensional, piecewise linear paths over the dyadic intervals of size $T2^{-N}$. 
\item This is the optimal finite dimensional approximation of Brownian motion with respect to the H\"older norm \eqref{eq:CielsielskiNorm1}. 
\item The support of the measure $\cL^{W^N}$ is Reflexive, so by Theorem \ref{theorem:ExistenceStationaryQuant} a stationary quantization exists. 
\end{itemize}

\subsubsection{Optimality of the Truncation}

We prove that the Truncation chosen in \eqref{eq:dfn:TruncatedBrownianMotion} is the optimal choice with respect to the $\alpha$-H\"older norm. This is an application of the results of \cite{bay2017karhunen}. 

\begin{proposition}
\label{pro:OptimalKarhunen}
Let $\cL^W$ be the law of Brownian motion over the Banach space $C^{\alpha, 0}([0,T]; \bR^{d'})$. Then the $d'\cdot 2^N$ dimensional projection $P:C^{\alpha, 0}([0,T]; \bR^{d'}) \to C^{\alpha, 0}([0,T]; \bR^{d'})$ that minimises the integral
$$
\bE\Bigg[ \Big\| W - P[W] \Big\|_{\alpha-\textrm{H\"ol}}^2 \Bigg],
$$
is the projection
$$
P[W]_t = \sum_{(p, m)\in \Lambda_N } W_{pm} G_{pm}(t).
$$
\end{proposition}

\begin{proof}
See Appendix \ref{sec:AppendixA}.  
\end{proof}

\subsubsection{Rate of Convergence of the Truncation}

We measure the rate of convergence for a truncated Brownian motion with respect to the $\alpha$-H\"older norm. We point out that the Banach space $C^{\alpha, 0}([0,T]; \bR^{d'})$ is not $K$-convex (see \cite{pisier1999volume}*{Definition 2.3}) so consequently the upper and lower bounds of the rate of convergence cannot be the same. 

%
%

\begin{proposition}
\label{pro:RateConverenceTruncationUpper}
Let $W$ be a Brownian motion as expressed in \eqref{eq:CielsielskiRepresentation} and let $W^{N}$ be truncated Brownian motion \eqref{eq:dfn:TruncatedBrownianMotion}. Then for $r>1$ we have 
\begin{equation}
c \cdot d' \cdot N^{1/2-\alpha} \cdot 2^{(\alpha-1/2)N}\leq \bE\Big[ \| W - W^{N}\|_{\alpha}^r\Big]^{1/r} \leq C\cdot d' \cdot \sqrt{N}\cdot 2^{(\alpha-1/2)N}, 
\end{equation}
where the constants $c$ and $C$ dependent only on $\alpha$ and $r$. 
\end{proposition}

\begin{proof}
Using Theorem \ref{thm:Ciecielski} and the methods of \cite{baldi1992some}, we have
\begin{align*}
\fB_W(\varepsilon) = -\log\Big( \bP\Big[ \| W\|_\alpha<\varepsilon\Big] \Big) \approx d' \cdot \Big(\tfrac{\varepsilon}{d'}\Big)^{\tfrac{-1}{1/2-\alpha}}
\quad
\mbox{and}
\quad
\fB_W(\varepsilon) \lesssim \fB_W(2\varepsilon). 
\end{align*}
Then by \cite{li1999approximation}*{Proposition 4.1}, this implies
$$
d' \cdot N^{1/2-\alpha} \cdot 2^{(\alpha-1/2)N} \lesssim \bE\Big[ \sup_{(p, m)\in \Lambda \backslash \Lambda_N} |W_{pm}|^2 2^{p(2\alpha-1)} \Big]^{1/2} \lesssim d'\cdot \sqrt{N}\cdot 2^{(\alpha-1/2)N},
$$
as $N\to \infty$ since $W^{N}$ is a $d\cdot 2^{N+1}$-dimensional Gaussian random variable. 

The Gaussian random variables $W-W^N$ can be dominated by $W$. By using a concentration inequality and a standard hypercontractivity argument, we can find a constant $C=C(r)$ such that
$$
\bE\Big[ \| W - W^N\|_\alpha^r\Big] \leq C(r) \bE\Big[ \| W - W^N\|_\alpha^2 \Big]^{\tfrac{r}{2}}. 
$$

Thus the rate of convergence in mean square is equivalent to the rate of convergence for any choice of $r$. 

\end{proof}

\subsubsection{Enhanced Truncated Brownian Motion}
\label{subsubsec:EnTrunc}

Finally, we prove that the rate of convergence of the \emph{enhanced truncated Brownian motion} to the \emph{enhanced Brownian motion} is the same when the process is lifted to a rough path and studied with respect to the inhomogeneous metric. 

The rate of convergence for an enhanced piecewise linear approximation of a Brownian motion has already been studied in \cite{friz2011convergence}. Our contribution is a sharper rate of convergence. 

\begin{proposition}
\label{pro:EnhancedRateConverg}
Let $N\in \bN$ and let $M\geq 2$. Let $\cL^{W^{N}}$ be the law of the truncated Brownian motion over the Banach space $C^{\alpha, 0}([0,T]; \bR^{d'})$. Then $\cL^{W^{N}}$ satisfies Assumption \ref{assumption:GaussianRegularity} hence $W^N$ can be lifted to an enhanced Gaussian rough path $\rw^{N}=S_M(W^N)$ taking values on the Group $G^M(\bR^{d'})$ for $M\geq 2$. Further, for the enhanced Brownian motion $\rw $ taking values in $G\Omega_\alpha(\bR^{d'})$, there exists a constant $C=C(M, d', \alpha)$ such that
\begin{equation}
\label{eq:pro:EnhancedRateConverg}
\bE\Big[ \rho_i(\rw_{s,t}^{N}, \rw_{s,t} )^2 \Big] \leq C N 2^{(2\alpha-1)N} |t-s|^i,
\end{equation}
where $i\in\{1, ..., M\}$ and $\rho_i$ is the tensor pseudo-metric \eqref{eq:pseudometric?} over $T^M(\bR^{d'})$. 
\end{proposition}

\begin{proof}
The case $i=1$ is immediate. We address $i=2$ briefly. For $j, k\in \{1, ..., d'\}$ and $j\neq k$
\begin{align*}
\bE\Bigg[&\Big| \int_s^t \langle W_{s, r}, e_j\rangle \circ d\langle W_r, e_k\rangle - \int_s^t \langle W_{s, r}^{N}, e_j\rangle \circ d\langle W_r^{N}, e_k\rangle \Big|^2 \Bigg]
\\
&\leq \int_s^t \int_s^t \cR_{\langle W-W^{N}, e_j\rangle} \begin{pmatrix}s,& s\\ u,& v\end{pmatrix} d\cR_{\langle W, e_k\rangle} (u, v)  
\\
&\quad +  \int_s^t \int_s^t \cR_{\langle W^{N}, e_j\rangle} \begin{pmatrix}s,& s\\ u,& v\end{pmatrix} d\cR_{\langle W^{N}-W, e_k\rangle} (u, v) 
\\
&\leq C  |t-s|^2 \cdot \bE\Big[ \| W - W^{N}\|_\alpha^2\Big] \cdot \bE\Big[ \| W\|_\alpha^2\Big].
\end{align*}
Compiling these terms by summing over $j$ and $k$ completes the $i=2$ case. 

For $i>2$, we argue by induction. For a word $A$ such that $|A|=i$ and letter $a\in \cA$, we have
\begin{align*}
\bE\Bigg[& \Big| \int_s^t \langle \rw_{s, r}, e_{(A, a)} \rangle \circ d\langle \rw_r, e_{(A,a)}\rangle - \int_s^t \langle \rw^N_{s, r}, e_{(A, a)} \rangle \circ d\langle \rw^N_r, e_{(A,a)}\rangle \Big|^2 \Bigg]
\\
&\leq \int_s^t \int_s^t \bE\Big[ \rho_i( \rw_{s, u}, \rw^N_{s, u}) \rho_i( \rw_{s, v}, \rw^N_{s, v}) \Big] d\cR_{\langle W, e_a\rangle}(u, v) 
\\
&\quad +\int_s^t \int_s^t \bE\Big[ \langle \rw_{s, u}, e_A\rangle \cdot \langle \rw_{s, v}, e_A\rangle \Big] d\cR_{\langle W^N-W, e_a\rangle}(u, v)
\\
&\leq C |t-s|^{i+1} \cdot \bE\Big[ \| W - W^{N}\|_\alpha^2\Big], 
\end{align*}
which implies the inductive hypothesis. 
\end{proof}

\begin{theorem}
\label{thm:EnhancedRateConverg}
Let $N\in \bN$ and $M\geq 2$. Let $r>1$. Let $\rw^{N}$ be the enhanced truncated Brownian motion and let $\rw$ be the enhanced Brownian motion over $G^M(\bR^{d'})$. Then
\begin{equation}
\label{eq:thm:EnhancedRateConvergInhomo}
\bE\Big[ \rho_{\alpha-\textrm{H\"ol}}\big( \rw, \rw^{N}\big)^r \Big]^{1/r} \lesssim \sqrt{N} \cdot 2^{(\alpha-1/2)N}
\end{equation}
as $N \to \infty$. 
Also
\begin{equation}
\label{eq:thm:EnhancedRateConvergHomo}
\bE\Big[ d_{\alpha-\textrm{H\"ol}}( \rw, \rw^{N})^r \Big]^{1/r} \lesssim \max \Big\{ \sqrt{N} 2^{(\alpha-1/2)N}, \big(\sqrt{N} 2^{(\alpha-1/2)N}\big)^{1/M}\Big\}. 
\end{equation}
\end{theorem}

\begin{proof}
Firstly, it should be clear that we have
$$
\bE\Big[ \rho_{\alpha-\textrm{H\"ol}}(\rw, \rId)^2\Big] < C\quad \mbox{and} \quad \bE\Big[ \rho_{\alpha-\textrm{H\"ol}}(\rw^N, \rId)^2\Big] < C. 
$$
Then, we apply \cite{friz2010multidimensional}[Theorem 15.24] with Proposition \ref{pro:EnhancedRateConverg} to get Equation \eqref{eq:thm:EnhancedRateConvergInhomo} in the case $r=2$. 

For \eqref{eq:thm:EnhancedRateConvergHomo}, we use the well known fact that the the identity operator is $\tfrac{1}{M}$-H\"older continuous from the space of rough paths paired with the Inhomogeneous metric to the space of rough paths paired with the homogeneous metric and $r=2$. 

Now for the case $r\neq 2$. Following \cite{riedel2017transportation}*{Corollary 3.2}, we can conclude that the pushforward of $d_\alpha(\rw, \rw^N)$ with respect to the measure $\cL^W$ has a Gaussian tail uniformly on $N$ since the covariance of $W-W^N$ can be dominated by the covariance of $W$. Then we use a hypercontractivity argument to conclude that
$$
\bE\Big[ d_\alpha(\rw, \rw^N)^r\Big] \leq C(r) \bE\Big[ d_\alpha(\rw, \rw^N)^2\Big]^{\tfrac{r}{2}}. 
$$
Thus the rate of convergence in mean square is equivalent to the rate of convergence for any choice of $r$. 
\end{proof}

\subsection{Quantization of Brownian Motion}


We perform a truncation to obtain a finite dimensional Gaussian that represents an optimal finite dimensional approximation of the Brownian motion. Here, we study how the choice of truncation affects the asymptotic rate of convergence of the quantization error. 

\begin{remark}
$\cL^{W^N}$ is a non-degenerate measure over the (finite dimensional) vector space $(\cH_N, \| \cdot \|_\alpha)$. Therefore by Theorem \ref{theorem:ExistenceStationaryQuant} we know that there exists a codebook $\fC_n=\{ \fc_1, ..., \fc_n\}$ and a partition $\hat{\fS}_n=\{ \hat{\fs}_1, ..., \hat{\fs}_n\}$ of $\cH_N$ such that the quantization $\hat{q}_n$ satisfies
$$
\bW^{(2)}_{\cH_N, \|\cdot\|_\alpha} \Bigg(\cL^{W^N}\Big|_{\cH_N}, \cL^{W^N}\circ \hat{q}^{-1}\Big|_{\cH_N}\Bigg) = \bE\Big[ \| W^N - \hat{q}_n(W^N)\|_\alpha^2 \Big]^{1/2} = \fE_{n}\Bigg.
$$

However, the measure $\cL^{W^N}$ is degenerate over the whole space $C^{\alpha, 0}([0,T]; \bR^{d'})$ so constructing an optimal quantization becomes analytically problematic.
\end{remark}

\begin{definition}
\label{dfn:TheQuantization}
Let $N\in \bN$ be fixed for the moment and let $n\in \bN$. Let $\cL^{\langle W,e_1\rangle}$ be the law of Brownian motion over $C^{\alpha, 0}([0,T]; \bR)$ and let $\cL^{\langle W^N, e_1\rangle}$ be the law of the 1-dimensional truncated Brownian motion with RKHS $\cH^{N, (1)}$. Let $\fC_n^{(1)}=\{ \fc_1^{(1)}, ..., \fc_n^{(1)} \}$ and $\hat{\fS}_n^{(1)}=\{ \hat{\fs}_1^{(1)}, ..., \hat{\fs}_n^{(1)} \}$ be the codebook and partition of the stationary quantization of $\cL^{\langle W^N, e_1\rangle}$ over $\cH^{N, (1)}$. 

Let $\fC_n:= \big( \fC_n^{(1)} \big)^{\times d'}$ and $\hat{\fS_n}:= \big( \hat{\fS}_n^{(1)} \big)^{\times d'}$. Thus $\fC_n$ and $\hat{\fS}_n$ form a quantization of the truncated Brownian motion over $\cH^N$ with independent components. Let $P_N:\cH \to \cH_N$ be the orthogonal projection and let us continuously extend $P_N$ to $\overline{\cH}^\alpha$. We define the new partition of $\overline{\cH}^\alpha$ to be 
\begin{equation}
\label{eq:dfn:TheQuantizationCodebook}
\fs_i:= (P_N)^{-1} \big[ \hat{\fs}_i\big], 
\quad
\fS_n:= \Big\{ \fs_1, ..., \fs_n \Big\}.
\end{equation}
Pairing the partition $\fS_n$ with the codebook $\fC_n$, we obtain a quantization for the truncated Brownian motion over $\overline{\cH}^{\alpha-\textrm{H\"ol}}$. 
\end{definition}

It is worth noting that the codebook $|\fC_n|=n^{d'}$. We should also emphasise that the quantization constructed in Definition \ref{dfn:TheQuantization} is not an optimal quantization of the measures $\cL^W$ or $\cL^{W^N}$ over the whole space. The reason for this approach is that this quantization exists and is solvable. 

\begin{lemma}
\label{lem:componentwiseInd}
Let $n,N \in \bN$. Let $\cL^W$ be the law of a Brownian motion over $C^{\alpha, 0}([0,T]; \bR^{d'})$ with quantization $q_n$ as defined in Definition \ref{dfn:TheQuantization}. 

Let $i\neq j \in \{ 1, ..., d'\}$. Then $\langle q_n(W), e_i\rangle$ and $\langle q_n(W), e_j\rangle$ are independent. 
\end{lemma}

\begin{proof}
For any two sets $C, D \in C^{\alpha, 0}([0,T]; \bR)$, we have
\begin{align*}
\bP\Big[& \langle q_n(W), e_i\rangle \in C, \langle q_n(W), e_j\rangle \in D\Big]
= \cL^W\Bigg[ \Big(\bigcup_{\substack{k\\\langle \fc_k, e_i\rangle \in C}} \fs_k\Big) \cap \Big( \bigcup_{\substack{l\\ \langle \fc_l, e_j\rangle \in D}} \fs_l \Big) \Bigg]
\\
&= \cL^{W^N}\Bigg[ \Big(\bigcup_{\substack{k\\\langle \fc_k, e_i\rangle \in C\cap \cH^N}} \hat{\fs}_k\Big) \cap \Big( \bigcup_{\substack{l\\ \langle \fc_l, e_j\rangle \in D\cap \cH^N}} \hat{\fs}_l \Big) \Bigg]
\\
&= \cL^{\langle W^N, e_1\rangle}\Bigg[ \bigcup_{\substack{k\\ \fc_k^{(1)}\in C\cap \cH^{N,1}}} \hat{\fs}_k^{(1)}  \Bigg] \cdot \cL^{\langle W^N, e_1\rangle} \Bigg[  \bigcup_{\substack{l\\ \fc_l^{(1)} \in D\cap \cH^{N,1}}} \hat{\fs}_l^{(1)}  \Bigg]
\\
&=\bP\Big[ \langle q_n(W), e_i\rangle \in C\Big] \cdot \bP\Big[ \langle q_n(W), e_j\rangle \in D\Big].
\end{align*}
\end{proof}
%

\subsubsection{Asymptotic rate of convergence for Quantization}

Next, we apply Theorem \ref{cor:AsymtoticRateConQuantErr} with Proposition \ref{pro:RateConverenceTruncationUpper} in order to demonstrate the rate of convergence of the quantization we construct. 

\begin{proposition}
\label{pro:ActualRateCon4Quantization}
Let $\cL^W$ be the law of a Brownian on $C^{\alpha, 0}([0,T]; \bR^{d'})$ and let $\cL^{W^N}$ be the law of the truncated Brownian motion. Choose $N$ to satisfy
\begin{equation}
\label{eq:thm:ActualRateCon4Quantization-N}
N\approx \frac{\cW \Big(\ln(2^{2\alpha-1}) \cdot \log(n)^{2\alpha-1}\Big) }{\ln(2^{2\alpha-1}) }, 
\end{equation}
where $\cW$ is the Lambert-W function (see \cite{weisstein2002lambert}), the inverse function of $y=xe^x$. 

Then $\forall r>1$, the quantization constructed in Definition \ref{dfn:TheQuantization} satisfies
\begin{align}
\label{eq:thm:ActualRateCon4Quantization}
\Bigg( \int_{C^{\alpha, 0}([0,T]; \bR^{d'})} \| x - q_n(x)\|_\alpha^r d\cL^W(x) \Bigg)^{1/r} 
\lesssim
d' \cdot \Big( \log(n) \Big)^{\alpha-1/2}
\end{align}
as $n\to \infty$. 
\end{proposition}

\begin{proof}

It should be clear that the partition as defined in Equation \eqref{eq:dfn:TheQuantizationCodebook} is not the collection of Voronoi sets generated by the codebook $\fC_n$ over $\overline{\cH}^{\alpha-\textrm{H\"ol}}$. Thus
$$
\bE\Big[ \big\| W - q_n(W) \big\|_\alpha^r\Big]^{1/r}\geq \Bigg( \int_{\overline{\cH}^{\alpha} } \min_{\substack{j=1, ..., n\\ \fc_j \in \fC}} \big\| x- \fc_j \|_\alpha^r d\cL^{W}(x) \Bigg)^{1/r}. 
$$
We can further improve this lower bound by minimizing over the all possible codebooks $\fC$ which yields the lower bound
$$
d'\Big( \log(n) \Big)^{\alpha-1/2} \lesssim \fE_{n^{d'},r}(\cL^{W}) \leq \bE\Big[ \big\| W - q_n(W) \big\|_\alpha^r\Big]^{1/r}. 
$$

For the upper bound, we apply Lemma \ref{lem:DivisionQuantization+Truncation} and Proposition \ref{pro:RateConverenceTruncationUpper} to get
\begin{align*}
\bE\Big[ \big\| W - q_n(W) \big\|_\alpha^r\Big]^{1/r} \leq& \bE\Big[ \big\| W^N - q_n(W^N) \big\|_\alpha^r \Big]^{1/r} + \bE\Big[ \big\| W - W^N \big\|_\alpha^r \Big]^{1/r}
\\
\lesssim& \fB_{ W^N}^{-1} \Big( \log(n) \Big) + \sqrt{N} \cdot 2^{(\alpha-1/2) N}. 
\end{align*}

By Theorem \ref{cor:AsymtoticRateConQuantErr}, we have asymptotic upper and lower bounds on the quantization error for both measures $\cL^W$ and $\cL^{W^N}$. 

Due to the nice way in which the truncation and the H\"older norm overlap, we have that
$$
\bP\Big[ \| W^N\|_\alpha \leq \varepsilon \Big] \geq \bP\Big[ \|W\|_\alpha <\varepsilon\Big],
$$
or equivalently
\begin{align*}
-\log\Big( \bP\Big[ \| W^N\|_\alpha \leq \varepsilon \Big] \Big)
=\fB_{W^N} (\varepsilon) 
\leq
\fB_W (\varepsilon)
= -\log\Big( \bP\Big[ \|W\|_\alpha <\varepsilon\Big]\Big).
\end{align*}
This is true for any choice of truncation level $N$. Taking the inverse of these bijective, increasing functions gives
$$
\fB_{W^N}^{-1}(n) \leq \fB_{W}^{-1}(n). 
$$
Thus, for any choice of $N\in \bN$, 
$$
\bE\Big[ \| W - q_n(W)\|_\alpha^r\Big]^{1/r} \lesssim d' \Big( \log(n) \Big)^{\alpha-1/2} + d'\sqrt{N} \cdot 2^{(\alpha-1/2)N}. 
$$
Finally, we note that the asymptotic relation of Equation \eqref{eq:thm:ActualRateCon4Quantization-N} is equivalent to
$$
\sqrt{N} \cdot 2^{(\alpha-1/2)N} \approx \Big( \log(n) \Big)^{\alpha-1/2}
$$
which yields the conclusion. 
\end{proof}


\begin{remark}
We know by results such as \cite{deuschel2017enhanced} that by sampling a Brownian motion in pathspace, the empirical law will be a  good approximation for the law of Brownian motion. 

The difference with this method is that sampling produces a convergence in measure type result. This is a \emph{deterministic} and not probabilistic result. 
\end{remark}

\subsubsection{Quantization for a Gaussian Rough Paths}

For this Section, we explore lifting our quantized Brownian motion to a rough path. Quantization for rough paths was first studied in \cite{pages2011convergence}. In their paper, the authors treat the law of Brownian motion as a measure over the Hilbert space $L^2([0,T]; \bR^{d'})$. In particular, as a measure over a Hilbert space the authors are able to obtain a stationary quantization, see \cite{LUSCHGY2002486}. The Karhunen Lo\`eve expansion is obtained using an expansion of trigonometric functions and the authors use well understood pathspace results to establish pointwise convergence of the paths followed by convergence in $p$-variation. 
To the best of our knowledge, this is the only work studying quantization in a rough path framework so this chapter is new and of independent interest. 

We perform quantization for a Brownian rough path with respect to the pathspace H\"older norm. Due to the nature of the $L^2$ norm with which the quantization is constructed in \cite{pages2011convergence}, the approximation with respect to the H\"older norm is far from optimal. By contrast, our approximation is arbitrarily close to optimal. In this Section, we prove that this remains true when the study is carried out with respect to the rough path H\"older norm. 

As proved in Lemma \ref{lem:StationaryRKHS}, the sets $\fC\subset \cH$ so have a canonical Young integral signature $\rc=S_M( \fc)$ for each $\fc \in \fC$.  

\begin{definition}
\label{dfn:TheQuantizationRough}
Let $M\geq 2$. Let $\cL^W$ be the law of a Brownian motion over $C^{\alpha, 0}([0,T]; \bR^{d'})$ and let $\cL^\rw$ be the law of the enhanced Brownian motion over $G\Omega_{\alpha}(\bR^{d'})$. Let $q_n$ be the sequence of quantizations as defined in Definition \ref{dfn:TheQuantization} for the truncated Brownian motion with $N$ chosen to satisfy Equation \eqref{eq:thm:ActualRateCon4Quantization-N} and codebooks $\fC_n$ and partitions $\fS_n$. 

Define the sets
$$
\rS=\Big\{\rs_1, ..., \rs_n \Big\}, \quad 
\rs_i:=\overline{\Big\{ \rh = S_M(h): h\in \fs_i\cap \cH \Big\}}^{\rho_{\alpha\textrm{-H\"ol}}}. 
$$
These form a partition over the space $G\Omega_\alpha(\bR^{d'})$ (up to boundary sets of measure 0). Similarly, define the codebook
$$
\rC=\{ \rc_1, ..., \rc_n\}, \quad \rc_i:=S_M[ \fc_i]. 
$$

By combining the enhanced codebook with the partition $\rS_n$, we obtain the enhanced quantization $\rQ_n: G\Omega_\alpha(\bR^{d'}) \to G\Omega_\alpha(\bR^{d'})$ 
\begin{equation}
\rQ_n( \rx ) = \rc_i \quad \mbox{ for } \rx\in \rs_i, \quad \rQ_n\Big( G\Omega_\alpha(\bR^{d'})\Big) = \rC_n. 
\end{equation}
\end{definition}

The next result is an extension of Proposition \ref{pro:ActualRateCon4Quantization} to the rough path setting. We follow the same methods as in Section \ref{subsubsec:EnTrunc}. 

\begin{proposition}
\label{pro:EnhancedQuantRateConverg}
Let $M\geq 2$. Fix $N, n\in \bN$. Let $\cL^{\rw}$ be the law of the enhanced Brownian motion. Then there exists a constant $C=C(M, d', \alpha)$ such that
\begin{equation}
\label{eq:pro:EnhancedQuantRateConverg}
\bE\Big[ \rho_i\big(\rw_{s,t}^{N}, \rQ_n(\rw^N)_{s,t} \big)^2 \Big] \leq C \Big( \log(n)\Big)^{2\alpha-1} |t-s|^i. 
\end{equation}
\end{proposition}

\begin{proof}
The case $i=1$ is already proved in Proposition \ref{pro:ActualRateCon4Quantization}. $i>2$ can be addressed via an induction argument as in Proposition \ref{pro:EnhancedRateConverg}. Therefore, we only prove the case $i=2$. Thus for $\cF_n$ equal to the $\sigma$-algebra generated by the partition of $\fS_n$, we have
\begin{align*}
\bE\Bigg[& \Big| \int_s^t \langle W_{s, r}^N, e_i\rangle d\langle W_{r}^N, e_j\rangle - \int_s^t \langle \bE[ W_{s, r}^N| \cF_n], e_i\rangle d \bE[ W_{r}^N| \cF_n], e_j\rangle  \Big|^2 \Bigg]
\\
\leq& 2 \int_s^t \int_s^t \bE\Bigg[ \Big\langle W_{s, r}^N - \bE[W_{s, r}^N| \cF_n], e_i \Big\rangle \cdot \Big\langle W_{s, u}^N - \bE[W_{s, u}^N| \cF_n], e_i \Big\rangle \Bigg] d\bE\Bigg[  \langle W_r^N, e_j\rangle \cdot  \langle W_u^N, e_j\rangle\Bigg]
\\
&+2 \int_s^t \int_s^t \bE\Bigg[  \langle W_{s, r}^N, e_i\rangle \cdot  \langle W_{s, u}^N, e_i\rangle\Bigg]\cdot d\bE\Bigg[ \Big\langle W_{r}^N - \bE[W_{r}^N| \cF_n], e_j \Big\rangle \cdot \Big\langle W_{u}^N - \bE[W_{u}^N| \cF_n], e_j \Big\rangle \Bigg], 
\\
\leq& (t-s)^{2} \bE\Bigg[ \Big\| W^N-\bE[W^N|\cF_n] \Big\|_{\alpha}^2 \Bigg] \cdot \bE\Big[ \| W^N\|_\alpha^2\Big] \leq C (t-s)^{2} \Big( \log(n) \Big)^{2\alpha-1},
\end{align*}
using Lemma \ref{lem:componentwiseInd} and the same Young Estimates as in Proposition \ref{pro:EnhancedRateConverg}. 
$$
\leq (t-s)^{2} \bE\Bigg[ \Big\| W^N-\bE[W^N|\cF_n] \Big\|_{\alpha}^2 \Bigg] \cdot \bE\Big[ \| W^N\|_\alpha^2\Big] \leq C (t-s)^{2} \Big( \log(n) \Big)^{2\alpha-1}. 
$$
\end{proof}

\begin{theorem}
\label{thm:ActualRateCon4QuantizationRough}
Let $r>1$. Let $\cL^W$ be the law of Brownian motion on $G\Omega_\alpha(\bR^{d'})$ and let $\cL^\rw$ be the law of the of the enhanced Brownian motion over $G\Omega_\alpha(\bR^{d'})$. Let $\rQ_n$ be the sequence of quantizations constructed in Definition \ref{dfn:TheQuantizationRough}. Then
\begin{equation}
\label{eq:thm:ActualRateCon4QuantizationRough}
\Bigg( \int_{G\Omega_\alpha(\bR^{d'})}  \rho_{\alpha-\textrm{H\"ol}; [0,T]}\Big( \rx, \rQ_n(\rx) \Big)^r d\cL^\rw(\rx) \Bigg)^{1/r} \lesssim \Big( \log(n) \Big)^{\alpha-1/2} .
\end{equation}
\end{theorem}

\begin{proof}
The lower bound of Equation \eqref{eq:thm:ActualRateCon4QuantizationRough} is actually immediate from Equation \eqref{eq:thm:ActualRateCon4Quantization}. The $\rho_{\alpha-\textrm{H\"ol}}$ metric can be lower bounded by the projection onto the first level of the Signature so that
$$
\bE\Big[ \rho_{\alpha-\textrm{H\"ol}; [0,T]}\Big( \rw, \rQ_n(\rw)\Big)^2 \Big] \geq \bE\Big[ \| W - q_n(W)\|_\alpha^2\Big]. 
$$
Also, by Theorem \ref{thm:EnhancedRateConverg}, we know the rate of convergence for
\begin{align*}
\bE\Big[ \rho_{\alpha-\textrm{H\"ol}}( \rw, \rw^{N})^r \Big]^{1/r} \lesssim& \sqrt{N} \cdot 2^{(\alpha-1/2)N}
\lesssim
\Big( \log(n)\Big)^{\alpha - 1/2},
\end{align*}
where $N$ is the dimension of the linear span of the codebook $\fC_n$ and the choice of Equation \eqref{eq:thm:ActualRateCon4Quantization-N} provides the second step. It is clear that
$$
\bE\Big[ \rho_{\alpha-\textrm{H\"ol}}(\rw, \rId)^2\Big] < C, \quad \bE\Big[ \rho_{\alpha-\textrm{H\"ol}}(\rw^N, \rId)^2\Big] < C. 
$$
We can then apply \cite{friz2010multidimensional}[Theorem A.13] with Proposition \ref{pro:EnhancedQuantRateConverg}. We remark that although this method has been used to prove the regularity of enhanced gaussian rough paths before, there is no part of this method that requires the Gaussian structures, only regularity properties in all moments. Thus 
$$
\bE\Big[ \rho_{\alpha-\textrm{H\"ol}; [0,T]}\Big( \rw^N, \rQ_n(\rw^N) \Big)^r\Big]^{1/r} \lesssim \Big( \log(n)\Big)^{\alpha- 1/2}. 
$$

\end{proof}


\section{Mean Field Rough Differential Equations}
\label{sec:MVSDEsrough}

In the first Section, we address the approach of \cite{CassLyonsEvolving} to solve McKean Vlasov Rough Differential Equations driven by a Brownian rough path. There, the authors prove Existence, Uniqueness and a Propagation of Chaos result for McKean Vlasov Rough Differential Equations of the form
\begin{equation}
\label{eq:MVRDE}
dX_t = \sigma(X_t)d\rw_t + b(X_t) d\gamma^\mu_t, \quad \mu = \cL^X, \quad X_0 = \xi, \quad t\in[0,T],
\end{equation}
where the path $\gamma^\mu_t = \int_0^t \mu_s ds$ represents the measure dependency in the drift term. \cite{CassLyonsEvolving} includes an explanation as to why the authors were unable to include a measure dependency in the diffusion terms. 

Rough Differential Equations with a measure dependency in the drift term are addressed in the more recent preprints \cite{2018arXiv180205882B} and \cite{2019arXiv180205882.2B}. We choose to present this work in the framework of \cite{CassLyonsEvolving} to reduce the complexity and avoid obfuscated algebraic argument. 

\subsection{Controls and the Accumulated p-Variation}

In this first Section, we establish a key condition for the integrability of our quantization. For notational simplicity, we denote $p=\tfrac{1}{\alpha}$. 

\begin{definition}
Let $\beta>0$ and suppose that $\omega: \Delta_T \to \bR^+$ is a control (recall Definition \ref{def:control1}). We define the Accumulated $\beta$-local $\omega$-variation by
$$
\rM_{\beta} (\omega):= \sup_{\substack{D=(t_i) \\ \omega(t_i, t_{i+1})\leq\beta}} \sum_{i: t_i\in D} \omega(t_i, t_{i+1}). 
$$
\end{definition}

The Accumulated $\beta$-local controls were first introduced in \cite{cass2013integrability}. We are interested in the specific case where the control is induced by a weakly geometric rough path. 

\begin{definition}
\label{def:AccumulatedPVariation}
Let $\beta>0$. Let $p>2$ and let $\rw\in G\Omega_{\alpha}(\bR^{d'})$. We define the Accumulated $\beta$-local $p$-variation of a geometric rough path to a non-negative function defined by
$$
\rM_{\beta, p}(\rw):= \rM_\beta (\omega_{\rw, p}). 
$$

We define the nondecreasing sequence $(\tau_i(\beta, p, \rw))_{i\in \bN}$ by
\begin{align}
\label{eq:def:GreedySequence}
\tau_0(\beta) = 0, \quad
\tau_{i+1}(\beta) =\inf\{ t>\tau_i(\beta); \| \rw \|_{p{-var}; [\tau_i(\beta), t]}^p\geq \beta\} \wedge T. 
\end{align}
This is sometimes referred to as the \emph{Greedy sequence}. Finally, we define the function $\rN_{\beta, p, [0,T]}:G\Omega_\alpha(\bR^{d'}) \to \bN\cup\{\infty\}$ given by
$$
\rN_{\beta, p, [0,T]}(\rw):=\sup\{ n\in \bN\cup \{0\}: \tau_n(\beta)<T\}. 
$$
\end{definition}

While stopping time arguments become problematic for McKean Vlasov Equations due to the presence of the measure dependency, we emphasise that the greedy sequence \eqref{eq:def:GreedySequence} is dependent only on the driving noise and not the solution. 

It is immediate from the definition that $\rM_{\beta, p}(\rw)\leq \|\rw\|_{p-var; [0,T]}^p$. However, when $\rw$ is a Gaussian rough path and $p>2$, we have $|W_{0,T}|^p\leq \|\rw_{0,T}\|_{cc}^p\leq \|\rw\|_{p-var; [0,T]}^p$ and $W_{0,T}\sim N(0, T)$ so 
$$
\bE\Big[ \exp\Big( \| \rw\|_{p-var; [0,T]}^p\Big)\Big] = \infty. 
$$

\begin{remark}
The Accumulated $p$-variation is a way of restricting the size of the $p$-variation in the  event that the $p$-variation becomes large. When the $p$-variation of a Gaussian is large, by far the most probable event is that there is a single large increment of the process. While the $p$-variation will increase proportionally to this steep increment, the Accumulated $\beta$-local $p$-variation is restricted to partitions where the increments cannot be larger than $\beta$ so the one increment does not make a proportional contribution. 
\end{remark}

The following Proposition is key to the construction of McKean Vlasov Rough Differential Equations driven by Gaussian processes. 

\begin{proposition}
\label{pro:CassLittererLyonsMomGen}
Let $\rw$ be a continuous, centred Gaussian rough path that satisfies Assumption \ref{assumption:GaussianRegularity}. Then $\forall \beta>0$, the random variable $\rM_{\beta, p}(\rw)$ has well defined Moment Generating Function 
$$
[0,\infty) \ni \theta \mapsto\bE\big[ \exp( \theta \rM_{\beta, p}(\rw)\big]<\infty. 
$$
\end{proposition}

\begin{proof}
See \cite{cass2013integrability}*{Theorem 6.3} for tail estimates of the law of the Accumulated $p$-variation. 
\end{proof}

The existence of a moment generating function for the Accumulated $p$-variation of the driving noise for the McKean Vlasov Rough Differential Equation is a key Assumption of \cite{CassLyonsEvolving}, see below. In order to prove propagation of chaos of a sequence of measures, the authors prove that the sequence of empirical measures each has a moment generating function and that the empirical laws converge weakly to the law of the driving noise. We verify the quantization also satisfies this condition:

\begin{lemma}[\cite{friz2010differential}]
\label{lem:LieAlgebraMartingaleFormula}
Let $\cL^W$ be the law of a Brownian motion over $C^{\alpha, 0}([0,T]; \bR^{d'})$. Let $h_1, ..., h_n$ be a collection of orthonormal elements of $\cH$. Let $W^n$ be a finite Karhunen Lo\`eve expansion of $W$ generated by the set $\{h_1, ..., h_n\}$ so that
$$
W^n = \bE\big[ W \big| \cF^n \big],
$$
where $\cF^n$ is the $\sigma$-algebra generated by the functionals $f_j = (i^*)^{-1} [h_j]$ for each $j=1, ..., n$. 

Then the Brownian rough path $\rw = S_2(W)$ satisfies the martingale formula
\begin{equation}
\label{eq:GaussianRPMartingaleForm1}
\bE\Big[ \log_{\boxtimes}( \rw_{s, t}) \Big| \cF^n\Big] = \log_{\boxtimes}( \rw_{s, t}^n),
\end{equation}
where $\rw^n_{s, t} = S_2(W^n)_{s, t}$. 
\end{lemma}

The martingale formula yields a very brief proof that the quantized Gaussians are adequately integrable. This first Lemma recasts the well known result mentioned earlier in Equation \eqref{eq:StationaryQuantConvex}. 

\begin{lemma}
\label{lem:ConvexityRPNorm}
Let $\cL^W$ be the law of Brownian motion over $C^{\alpha, 0}([0,T]; \bR^{d'})$. Let $\cF$ be a sub-$\sigma$ algebra of the Borel sigma algebra over $C^{\alpha, 0}([0,T]; \bR^{d'})$ that is component-wise conditionally independent. Define $\tilde{W}=\bE[ W| \cF]$. Let $\rw$ be the Gaussian rough path of $\cL^W$ and $\tilde{\rw}$ be the lift of the random variable $\tilde{W}$ to a rough path. 

Then, for a constant $C_1=C_1(d', p)$ dependent only on $d'$ and $p$, we have
$$
\| \tilde{\rw} \|_{p-var; [0,T]}^p \leq C_1 \bE\Big[ \|\rw\|_{p-var; [0,T]}^p \Big| \cF \Big]. 
$$
\end{lemma}

\begin{proof}
Firstly, we work with the homogeneous norm \eqref{eq:HomoNorm} for $G^2(\bR^{d'})$ rather than the Carnot Caratheodory norm in order to evaluate the increments explicitly.

By component-wise conditional independence (for the 2nd equality) we have
\begin{align*}
\| \tilde{\rw}& \|_{p-var; [0,T]}^p 
\\
=& \sup_{D=(t_i)} \sum_{i: t_i\in D}\left( \sum_{j=1}^{d'} \Big| \langle \bE[ W_{t_i, t_{i+1}}| \cF] , e_j\rangle \Big| + \sum_{\substack{j,k=1\\ j\neq k}}^{d'} \Big| \int_{t_i}^{t_{i+1}} \langle \bE[ W_{t_i, u} | \cF] , e_j\rangle d \langle \bE[ W_{u}| \cF] , e_k\rangle \Big|^{1/2} \right)^p
\\
=&\sup_{D=(t_i)} \sum_{i: t_i\in D}\left( \sum_{j=1}^{d'} \Big| \langle \bE[ W_{t_i, t_{i+1}}| \cF] , e_j\rangle\Big| + \sum_{\substack{j,k=1\\ j\neq k}}^{d'} \Big| \bE\Big[\int_{t_i}^{t_{i+1}} \langle W_{t_i, u}, e_j\rangle d \langle W_{u}, e_k\rangle\Big| \cF\Big] \Big|^{1/2} \right)^p
\\
\leq&d'^{(2p-2)/p} \sup_{D=(t_i)} \sum_{i: t_i\in D}\bE\left[\Big( \sum_{j=1}^{d'} \Big| \langle W_{t_i, t_{i+1}}, e_j\rangle\Big| + \sum_{\substack{j,k=1\\ j\neq k}}^{d'} \Big| \int_{t_i}^{t_{i+1}} \langle W_{t_i, u}, e_j\rangle d \langle W_{u}, e_k\rangle\Big|^{1/2}\Big)^p \Big| \cF\right]
\\
\leq& d'^{(2p-2)/p} \bE\Big[ \|\rw\|_{p-var; [0,T]}^p \Big| \cF\Big],
\end{align*}
where we use a finite dimensional norm equivalence for the first inequality. There is a further multiplicative constant that appears from translating this result back to the Carnot Caratheodory norm which is dependent only on $d'$. 
\end{proof}

This result does not follow immediately via the same convexity argument used in Equation \eqref{eq:StationaryQuantConvex} because the \emph{Expectation of a Group element may not be a Group element itself}. 

\begin{proposition}
\label{pro:AccumulatedControlTrick}
Let $n, N\in \bN$. Let $\cL^W$ be the law of a Brownian motion on $C^{\alpha, 0}([0,T]; \bR^{d'})$ and let $W^N$ be the truncated Brownian motion. From Definition \ref{def:AccumulatedPVariation}, let $\tau_i(\beta)$ be the greedy sequence of the Brownian rough path $\rw$, let $\overline{\tau}_i(\overline{\beta})$ be the greedy sequence of the enhanced truncated Brownian motion $\rw^N = S_2(W^N)$ and let $\tilde{\tau}_i(\tilde{\beta})$ be the greedy sequence of the enhanced quantization $\rQ_n(\rw)$ as introduced in Definition \ref{dfn:TheQuantizationRough}. Let $\overline{\beta}=C_1 \beta$ and $\tilde{\beta} = C_1 \overline{\beta}$ where $C_1$ is the constant introduced in Lemma \ref{lem:ConvexityRPNorm}. 

Let $\rN_{\beta, p, [0,T]}(\rw)$, $\overline{\rN}_{\overline{\beta}, p, [0,T]}(\rw^N)$ and $\tilde{\rN}_{\tilde{\beta}, p, [0,T]}(\rQ_n(\rw))$ be the number of elements of each of the respective greedy sequences over the interval $[0,T]$. Then 
$$
\tilde{\rN}_{\tilde{\beta}, p, [0,T]}(\rQ_n(\rw)) \leq \overline{\rN}_{\overline{\beta}, p, [0,T]}(\rw^N) \leq \rN_{\beta, p, [0,T]}(\rw).
$$ 
\end{proposition}

\begin{proof}
This proof relies on the choice of quantization, and we choose $q(W^N)$ to be the optimal quantization of the finite dimensional Gaussian random variable $W^N$ as a measure over the set $\cH^N$ with independent spatial components, see Lemma \ref{lem:componentwiseInd}. Let $\tilde{\cF}$ be the $\sigma$-algebra generated by the partition of the quantization $\tilde{\cF}=\sigma(\fS)$ and let $\overline{\cF}$ be the cylindrical sigma algebra generated by the functionals $(i^*)^{-1}[\cH_N]$. Then we have $q(W^N) = \bE[ W^N| \tilde{\cF}]$ and $W^N = \bE[ W| \overline{\cF}]$. 

By Lemma \ref{lem:ConvexityRPNorm}, we therefore have that for any subinterval $[s, t]$ 
$$
\| \rQ_n(\rw) \|_{p-var; [s,t]}^p \leq C_1 \bE\Big[ \| \rw^N \|_{p-var; [s, t]}^p \Big| \tilde{\cF}\Big], \quad \| \rw^N \|_{p-var; [s,t]}^p \leq C_1 \bE\Big[ \| \rw \|_{p-var; [s, t]}^p \Big| \overline{\cF}\Big].
$$
In particular, for the intervals $[0, \overline{\tau}_1(\overline{\beta})]$ and $[0, \tau_1(\beta)]$ we have
\begin{align*}
\| \rQ_n(\rw) \|_{p-var; [0, \overline{\tau}_1(\overline{\beta})]}^p \leq C_1 \bE\Big[ \| \rw^N \|_{p-var; [0, \overline{\tau}_1(\overline{\beta})]}^p \Big| \tilde{\cF}\Big] = C_1 \overline{\beta}, 
\\
\| \rw^N \|_{p-var; [0,\tau_1(\beta)]}^p \leq C_1 \bE\Big[ \| \rw \|_{p-var; [0, \tau_{1}(\beta)]}^p \Big| \overline{\cF}\Big] = C_1 \beta. 
\end{align*}
However, by definition we also have $\| \rQ_n(\rw) \|_{p-var; [0,\tilde{\tau}_1(\tilde{\beta})]}^p = \tilde{\beta}$ and $\| \rw^N \|_{p-var; [0,\overline{\tau}_1(\overline{\beta})]}^p=\overline{\beta}$, so we conclude that $0<\tilde{\tau}_1(\tilde{\beta}) \leq \overline{\tau}_1(\overline{\beta}) \leq \tau_1(\beta)$. 

Next, arguing via induction we suppose that $\tilde{\tau}_k(\tilde{\beta}) \leq \overline{\tau}_k(\overline{\beta}) \leq \tau_k(\beta)$. Then
\begin{align*}
\| \rw^N\|_{p-var; [\overline{\tau}_k(\overline{\beta}), \tau_{k+1}(\beta) \vee \overline{\tau}_k(\overline{\beta})]}^p 
\leq& C_1 \bE\Big[ \| \rw \|_{p-var; [\overline{\tau}_k(\overline{\beta}), \tau_{k+1}(\beta) \vee \overline{\tau}_k(\overline{\beta})]}^p \Big| \overline{\cF}\Big] 
\\
\leq& C_1 \bE\Big[ \| \rw \|_{p-var; [\tau_k(\beta), \tau_{k+1}(\beta)]}^p \Big| \overline{\cF}\Big] = C_1 \beta, 
\\
\| \rQ_n(\rw)\|_{p-var; [\tilde{\tau}_k(\tilde{\beta}), \overline{\tau}_{k+1}(\overline{\beta}) \vee \tilde{\tau}_k(\tilde{\beta}) ]}^p 
\leq& C_1 \bE\Big[ \| \rw^N \|_{p-var; [\tilde{\tau}_k(\tilde{\beta}), \overline{\tau}_{k+1}(\overline{\beta}) \vee \tilde{\tau}_k(\tilde{\beta}) ]}^p \Big| \tilde{\cF}\Big] 
\\
\leq& C_1 \bE\Big[ \| \rw^N \|_{p-var; [\overline{\tau}_k(\overline{\beta}), \overline{\tau}_{k+1}(\overline{\beta})]}^p \Big| \tilde{\cF}\Big] = C_1 \overline{\beta}. 
\end{align*}
However, $\| \rw^N\|_{p-var; [\overline{\tau}_k(\overline{\beta}), \overline{\tau}_{k+1}(\overline{\beta}) ]}^p = \overline{\beta}$ and $\| \rQ_n(\rw)\|_{p-var; [\tilde{\tau}_k(\tilde{\beta}), \tilde{\tau}_{k+1}(\tilde{\beta}) ]}^p = \tilde{\beta}$ so we conclude $\tilde{\tau}_{k+1}(\tilde{\beta}) \leq \overline{\tau}_{k+1}(\overline{\beta}) \leq \tau_{k+1}(\beta)$.

Next, suppose that $\rN_{\beta, p, [0,T]}(\rw)=k$ for some $k\in \bN$. Then $T< \tau_{k+1}(\beta)\geq \overline{\tau}_{k+1}(\overline{\beta}) \geq \tilde{\tau}_{k+1}(\tilde{\beta}) $. Thus $k$ is an upper bound for $\overline{\rN}_{\overline{\beta}, p, [0,T]}(\rw^N)$ and $\tilde{\rN}_{\tilde{\beta}, p, [0,T]}(\rQ_n(\rw))$. 
\end{proof}

Finally, we establish the uniform integrability of the quantizations. 

\begin{proposition}
Let $\cL^\rw$ be the law of an enhanced Brownian motion and let $\cL^\rw \circ \rQ_n^{-1}$ be the law of the quantized Brownian motion. 

Then the Moment Generating function of the Accumulated $p$-variation of $\rQ_n(\rw)$ is well defined and bounded by the Moment Generating function of the Accumulated $p$-variation of $\rw$. 
\end{proposition}

\begin{proof}
From \cite{cass2013integrability}*{Proposition 4.11}, we have
$$
\beta \rN_{\beta, [0,T]}( \omega) \leq \rM_{\beta}( \omega) \leq \beta\Big( 2\rN_{\beta, [0,T]}( \omega) +1 \Big),
$$
for any control $\omega$ so the existence of a Moment Generating Function for $\rN$ is equivalent to the existence of a Moment Generating Function for $\rM$. 

Therefore, by Proposition \ref{pro:CassLittererLyonsMomGen}, we have that $\forall \theta, \beta>0$ that
$$
\bE\Big[ \exp\Big( \theta \rN_{\beta, p, [0,T]}(\rw) \Big) \Big] <\infty. 
$$
Applying Proposition \ref{pro:AccumulatedControlTrick}, we get that
$$
\exp\Big( \theta \tilde{\rN}_{(C_1)^2\beta, p, [0,T]}\big(\rQ_n(\rw) \big) \Big) \leq \exp\Big( \theta \rN_{\beta, p, [0,T]}(\rw) \Big). 
$$
We take expectations to conclude. 
\end{proof}

\subsection{Existence, Uniqueness and the Occupation Measure Path}

In this Subsection, we overview some of the key details of \cite{CassLyonsEvolving} to establish the link between particle systems and McKean Vlasov Equations and the existence and uniqueness of the solution law of McKean Vlasov Equations. 

The space of measures $\mu$ over the metric space $(E, d, \cE)$ is not a Banach space. However, a measure can be thought of as a functional over the space of Lipschitz functions on $E$. 

\begin{definition}
For $\mu\in \cP_2(E)$, we define $\gamma^\mu \in \lip_*^1(E)^*$ to be the linear functional such that for any $f\in \lip_*^1(E)$, 
$$
\gamma^\mu[f] = \int_E f d\mu. 
$$
Similarly, for a collection of measures $(\mu_t)_{t\in [0,T]}$, we define the \emph{Occupation measure path} $\gamma_t^\mu$. 
\end{definition}

First introduced in \cite{CassLyonsEvolving}, it is further proved that for the law of an SDE $\mu_t$, the Occupation measure path $\gamma^\mu$ is bounded variation in the Banach norm and so has a canonical Young Signature. The existence and uniqueness of a solution to equation \eqref{eq:ThetasEquation} comes immediately from \cite{friz2010multidimensional}*{Chapter 12}. 

\begin{assumption}
\label{Ass:Lipschitz:b+sigma}
Let $\varsigma>\tfrac{1}{\alpha}, \gamma>1$ and $M=\lfloor \tfrac{1}{\alpha}\rfloor$. Let 
$$
\sigma\in \lip^\varsigma \Big(\bR^d , L( \bR^{d'}, \bR^d)\Big) \quad \mbox{and}\quad b\in \lip^\gamma \Big(\bR^d, L(\lip_*^1(G^{ M}(\bR^d))^*, \bR^d)\Big).
$$
\end{assumption}

\begin{definition}
\label{Def:b.sigma}
Suppose $b$ and $\sigma$ satisfy Assumption \ref{Ass:Lipschitz:b+sigma}. Let $\mu \in \cP_1(G\Omega_\alpha(\bR^d)$, $\xi \in \bR^d$ and $\rw = G\Omega_\alpha(\bR^{d'})$. 

Then the operator $\Theta_{b, \sigma}: \cP_2\Big(G\Omega_\alpha(\bR^d)\Big) \times \bR^d \times G\Omega_\alpha(\bR^{d'}) \to G\Omega_\alpha(\bR^d)$ maps $(\mu, \xi, \rw)$ to the rough path that is the solution of the Rough Differential Equation
\begin{align}
\label{eq:ThetasEquation}
&d\rx_t = b(X_t)d\gamma_t^\mu + \sigma(X_t) d\rw_t, \quad \rx_0=\xi, 
\\
\nonumber
&(\mu, \xi, \rw) \mapsto \Theta(\mu, \xi, \rw)= \rx. 
\end{align}
\end{definition}

\subsubsection{Particle Approximations and Finite Support Laws}

Firstly, we address the existence and uniqueness of a solution to the system of interacting particles that the McKean Vlasov equation models. Let $\fC$ be a codebook for a quantization of the law of the Brownian motion $\cL^W$ as a measure over the Banach space $C^{\alpha, 0}([0,T]; \bR^{d'})$ containing $n$ elements $\fh^j$. Each $\fh^j$ is a RKHS path. Associated to each path is a component of the probability vector $\fp=(\fp_j)$ such that $\fp_j =  \cL^W(\fs_j)$ where $\fs_j\in \fS$ is the element of the partition associated to $\fh^j$. 

By the nature of $\cH$, we know that each path $\fh^j$ is a 1-variation path. Hence one can construct a canonical lift from $\fh^j$ to a rough path $\rh^j$ using Young Integration over the interval $[0,T]$. Thus for $t\in[0,T]$ where $M$ is the largest integer such that $M\alpha<1$ we have
$$
\rh_t^j = S_{M} ( \fh^j) _{0, t}.
$$

We know that $n$ is a finite integer, so we can denote the single path $\fh:=\times_{j=1}^n \fh^j$ which takes values in $\bR^{d'\times n}$. This path is still 1-variation with respect to the canonical norm on $\bR^{d'\times n}$. Therefore, we can similarly construct
$$
\rh_t = S_{M} \Big( \times_{j=1}^n \fh^j \Big) _{0, t}. 
$$

For clarity, we emphasise that this is a rough path taking values in $T^{M}\Big( \bR^{d'\times n}\Big)$ and it is not the same as $\oplus_{j=1}^n \rh^j = \oplus_{j=1}^n S_M(\fh^j)$ which takes values in $T^{M}( \bR^{d'} )^{\oplus n}$. 
\medskip

When working on the tensor algebra $T^M(V)$. We refer to the Alphabet $\cA$, which in the case $V=\bR^{d'}$, is the letters $\{1, ..., d'\}$. However, when working on the tensor algebra $T^M\Big( \bR^{d'\times n}\Big)$, we have the Alphabet $\cA$ containing all the pairs $\big\{ (i, j); i\in \{1, ..., d'\}, j\in \{1, ..., n\}\big\}$. We will also refer to $\cA^j$, the Subalphabet containing all pairs $\big\{ (i, j); i\in\{1, ..., d' \}\big\}$. Key to the following result is that the Subalphabets $\cA^j$ form a partition of the Alphabet $\cA$. 

\begin{lemma}
Let $V$ be a vector space with finite Alphabet $\cA$ and suppose that $\cA$ can be partitioned into a finite number of Subalphabets denoted by $\cA^j$. Define
$$
I^M(V):= \Big\{ h\in T^M(V): \langle h, e_I\rangle =0, \forall I \mbox{ a word with letters in } \cA \mbox{ s.t. } \exists j \mbox{ where $I$ is a word of } \cA^j\Big\}. 
$$
Then $I^M(V)$ is a closed ideal of the Lie Algebra $P^M(V)$. 
\end{lemma}

\begin{proof}
We verify that for $h_1\in I^M(V)$ and $h_2\in P^M(V)$ that $[h_1, h_2]_{\boxtimes} \in I^M(V)$. 

Let $I$ be a word that has the property that $\exists j$ such that $I$ is also a word of $\cA^j$. We denote
$$
\Delta e_I = \sum_{I_1 I_2 = I} e_{I_1} \otimes e_{I_2}
$$
using ``Sweedler'' notation and $I_1 I_2$ as being word concatenation. If $I$ is a word with letters in $\cA^j$ then any subword of $I$ is also a word with letters in $\cA^j$. 

Therefore, for $h_1 \in I^M(V)$ and $h_2\in P^M(V)$
$$
\langle h_1 \boxtimes h_2, e_I\rangle = \langle h_1 \otimes h_2, \Delta e_I \rangle = \sum_{I_1I_2 = I} \langle h_1, e_{I_1}\rangle \cdot \langle h_2, e_{I_2} \rangle = \sum_{I_1I_2 = I} 0 \cdot \langle h_2, e_{I_2} \rangle = 0. 
$$

Similarly $\langle h_2 \boxtimes h_1, e_I\rangle = 0$, so naturally
$$
\Big\langle [ h_1, h_2]_{\boxtimes}, e_I\Big\rangle = 0. 
$$
\end{proof}

Given an Ideal of a Lie Algebra, one can obtain a normal subgroup of the associated Lie Group by taking exponentials. Thus define
\begin{equation}
\label{eq:NormalSubGroup}
K^M(V):= \exp_{\boxtimes} \Big( I^M(V)\Big),
\end{equation}
and consider the quotient group $G^M(V)/ K^M(V)$. There is a canonical isomorphism that maps this quotient group to $\oplus_{j} G^M(V^j)$ where $V^j$ is the vector space with Alphabet $\cA^j$. 

In order to study the system of interacting particle equations for \eqref{eq:MVRDE}, we consider the following drift and diffusion terms. Before that, we introduce a notational convenience in order to distinguish between elements of $\bR^d$ and $\bR^{d\times n}$. Recall that for $i\in \cA$, $e_i$ is the unit vector in the vector space with Alphabet $\cA$. We denote $Y\in \bR^{d\times n}$ and $\langle Y, e_{(\cdot, m)}\rangle \in \bR^d$ to be the canonical projection of $Y$ where $m\in\{1, ..., n\}$. 

\begin{definition}
\label{dfn:BandSigma}
Let $b$ and $\sigma$ satisfy Assumption \ref{Ass:Lipschitz:b+sigma}. Let $\fp = (\fp_k)_{k=1, ... n} \in \fP$. Let $B: \bR^{d\times n} \to \bR^{d\times n}$ and $\Sigma: \bR^{d\times n} \to L\big( \bR^{d'\times n}, \bR^{d\times n}\big)$ be defined by
\begin{align*}
B(X):=& \bigoplus_{m=1}^n \left( b\big( \langle X, e_{(\cdot, m)}\rangle\big) \Big[ \sum_{k=1}^n \fp_k \delta_{\langle X, e_{(\cdot, k)}\rangle} \Big] \right), 
\\
\Sigma(X):=& \diag_{m=1, ..., n}\Big( \sigma\big( \langle X, e_{(\cdot, m)} \rangle \big) \Big). 
\end{align*}
\end{definition}

Let $\tilde{\rw}_t^k \in G\Omega_\alpha(\bR^{d'})$ for each $k\in\{1, ..., n\}$. Let $\tilde{\rw} = \bigoplus_{k=1}^n \tilde{\rw}^k$ be the rough path taking values in the quotient group $G^{M}\big( (\bR^{d'\times n}\big) / I^{M}\big( (\bR^{d'\times n}\big)$. Let $\bX_t$ be the controlled rough that solves the Rough Differential Equation
\begin{equation}
\label{eq:ParticleSystem1}
dX_t = B(X_t) dt + \Sigma(X_t) d\tilde{\rw}_t, \quad X_0 = \xi^{\times n},
\end{equation}
taking values in $\bR^{d\times n}$. By the properties of $b$ and $\sigma$ from Definition \ref{dfn:BandSigma}, we have that $B\in \lip^\gamma\Big( (\bR^{d\times n}\Big)$ and $\Sigma\in \lip^\varsigma\Big( (\bR^{d\times n}, L\big( \bR^{d'\times n}, \bR^{d\times n}\big) \Big)$. Therefore, the existence of a solution to Equation \eqref{eq:ParticleSystem1} is standard. 

Next we introduce a product on the space of vector fields from $U$ into $T^M(V, U)$ designed to simplify the representation of a controlled rough path. 

\begin{definition}
Let $V$ and $U$ be vector spaces. Let $i, j\in \bN$. For differentiable Vector fields $F:U \to L(V^{\times i}, U)$ and $G:U \to L( V^{\times j}, U)$, we define the operation $\star$ such that $F\star G:U \to L(V^{\times (i +j)}, U)$ by
\begin{align}
\nonumber
F\star G(u)\Big[ v_1, ..., v_j, v_{j+1}, ..., v_{j+i}\Big] =& \left(\lim_{\varepsilon \to 0} \frac{F\big( u + \varepsilon G(u)[v_1,..., v_j]\big) - F\big(u\big)}{\varepsilon}\right)[v_{j+1}, ..., v_{j+i}]
\\
\label{eq:TensorStar1}
=& DF(u)\Big( G(u)[v_1, ...,v_j] \Big) [v_{j+1}, ..., v_{j+i}]
\end{align}
\end{definition}

It is a natural observation to make that the controlled rough path $\bX$ that represents the solution to Equation \eqref{eq:ParticleSystem1} is equal to
\begin{equation}
\label{eq:ControlledXFormula}
\bX_s = \Big( X_s, \Sigma(X_s), \Sigma\star\Sigma(X_s), ..., \Sigma^{\star (M-1)}(X_s) \Big),  \quad s\in [0,T]. 
\end{equation}

\begin{lemma}
\label{lem:FstarG}
Let $V$, $U$ be vector spaces with alphabets $\cA$ and $\hat{\cA}$. Suppose that $V=\oplus_{j=1}^n V^j$ and $U=\oplus_{j=1}^n U^j$ so that $\cA$ and $\hat{\cA}$ can be partitioned into a collection of $n$ subalphabets $\cA^j$ and $\hat{\cA}^j$ for $j=1, ..., n$. 

For $k, l\in \bN$, let $F:U \to L(V^{\oplus k}, U)$ and $G:U \to L(V^{\oplus l}, U)$, suppose that there exist $f^j: U^j \to L\Big( (V^j)^{\oplus k}, U^j\Big)$ and $g^j: U^j \to L\Big( (V^j)^{\oplus k}, U^j\Big)$ such that we have the representation
\begin{align*}
F(u) = \diag_{j=1, ..., n} \Big( f^j ( P_{U^j}[u] ) \Big), 
\quad
G(u) = \diag_{j=1, ..., n} \Big( g^j ( P_{U^j}[u] ) \Big). 
\end{align*}
Suppose that $F$ is differentiable. Then $F \star G$ has the representation
\begin{equation}
\label{eq:lem:FstarG}
F \star G(u) = \diag_{j=1, ..., n} \Big( Df( P_{U^j}[u]) \times g( P_{U^j}[u]) \Big).
\end{equation}
\end{lemma}

\begin{proof}
For fixed $m\in\{1, ..., n\}$, let $u_m\in U^m$ and let $I$ be a word of the subalphabet $\cA^m$ such that $I=(I_1, I_2)$ where $|I_1|=k$ and $|I_2|=l$. 

Outside of this scenario, all derivatives will be 0 by construction. 
\end{proof}

We know that by Theorem \ref{thm:LiftofControlledPath}, the controlled rough path $\bX$ can be lifted to a rough path. Our next result, the main result of this Section and similar to one found in \cite{CassLyonsEvolving}, ensures the choice of lift does not affect the final solution to our equations. 

\begin{theorem}
\label{thm:ExtensionTheorem1}
For $j=1, ..., n$, let $\rw^j\in G\Omega_{\alpha}(\bR^{d'})$ and define $\rw = \oplus_{j=1}^n \rw^j$. Let $\tilde{\rw}$ be the extension of $\rw$ to $G\Omega_\alpha\Big( \bR^{d'\times n} \Big)$. Let $B$ and $\Sigma$ be as defined in Definition \ref{dfn:BandSigma} and let $\bX$ be the unique controlled rough path that solves the Rough Differential Equation \eqref{eq:ParticleSystem1}. 

Let $\rx\in G\Omega_\alpha\Big( \bR^{d \times n}\Big)$ be the lift of $\bX$ as constructed in Equation \eqref{eq:thm:LiftofControlledPath}. Then $\rx$ is dependent on $\rw$ but not $\tilde{\rw}$. 
\end{theorem}

\begin{proof}
See Appendix \ref{sec:AppendixA}. 
\end{proof}

\subsubsection{Existence and Uniqueness}

For this section, we focus on the approach of \cite{CassLyonsEvolving}. Firstly, we introduce some of the notation and operators used in this paper to construct different elements for solving our McKean Vlasov equation. The methods and results of \cite{2018arXiv180205882B} which are further explored in \cite{2019arXiv180205882.2B} and \cite{2019arXiv190700578B} are not used here. 

\begin{definition}
\label{dfn:FixedPointOperator}
Let $b$ and $\sigma$ satisfy Assumption \ref{Ass:Lipschitz:b+sigma}. Let $\cL \in \cP_1\big(G\Omega_\alpha(\bR^{d'})\big)$ and $\mu \in \cP_1(G\Omega_\alpha(\bR^d)\big)$ be probability measures. Then define the map $\Psi_\cL: \cP_2\big(G\Omega_\alpha(\bR^d)\big) \to \cP_2\big(G\Omega_\alpha(\bR^d)\big)$ by
\begin{equation}
\label{eq:dfn:FixedPointOperator}
\Psi_\cL(\mu) = \cL \circ \Theta_{b, \sigma}(\mu, \xi, \cdot)^{-1}. 
\end{equation}
\end{definition}

The fixed point of the operator $\Psi_\cL$ will be the law of the solution to the McKean Vlasov Equation \eqref{eq:MVRDE} where the law of the driving noise $\rw$ is given by $\cL$. 

\begin{assumption}
\label{Assumption:CassLyons1}
Let $\varsigma>\tfrac{1}{\alpha}>1$ and $\gamma>1$. Suppose that 
\begin{enumerate}
\item The measure $\cL^\rw \in \cP_2(G\Omega_\alpha(\bR^{d'}))$ satisfies that for any $ \theta\geq0$
\begin{equation}
\label{eq:Assumption:CassLyons1}
\int_{G\Omega_\alpha(\bR^{d'})} \exp\Big( \theta \rM_{1, [0,T]} (\omega_\rx) \Big) d\cL^\rw (\rx)<\infty, 
\end{equation}
\item The functions $b$ and $\sigma$ satisfy Assumption \ref{Ass:Lipschitz:b+sigma}. 
\end{enumerate}
\end{assumption}

\begin{theorem}[\cite{CassLyonsEvolving}]
Suppose Assumption \ref{Assumption:CassLyons1} holds. Then the operator $\Psi_{\cL^\rw}$ is a contraction operator with fixed point equal to the law of the solution to the McKean Vlasov Equation \eqref{eq:MVRDE}. 
\end{theorem}

Hence there exists a unique solution to the Rough Differential Equation \eqref{eq:MVRDE}. 

\subsection{Propagation of Chaos and Quantization}

The final result of \cite{CassLyonsEvolving} is to prove continuity of the map from the law of the driving noise to the law of the McKean Vlasov Equation. This is framed within the narrative of ``Propagation of Chaos''. We exploit this result to show that the law of the associated particle systems of our quantizations converge to the true law of the McKean Vlasov Equation. 

\begin{definition}
Let $K:(0, \infty) \to (0, \infty)$ be a monotone increasing real valued function. Define the collection of measures
\begin{align*}
\cP_K\Big( G\Omega_{\alpha}(\bR^{d'})\Big):=\Bigg\{ \cL\in \cP_1\Big( G\Omega_\alpha(\bR^{d'})\Big):& \forall \theta\in(0,\infty)
\\
&\int_{G\Omega_\alpha(\bR^{d'})} \exp\Big( \theta \rM_{1, [0,T]} (\omega_\rx) \Big) d\cL(\rx)\leq K(\theta) \Bigg\}
\end{align*}
paired with the topology of weak convergence generated by the rough path H\"older norm. 
\end{definition}

A natural way to think about this collection of measures is the law of all rough paths such that the moment generating function of the Accumulated $\tfrac{1}{\alpha}$-variation is dominated by the function $K$. 

\begin{proposition}
\label{pro:ContinuityXi}
Suppose Assumption \ref{Assumption:CassLyons1} is satisfied. Suppose additionally that there exists a monotone increasing function $K:(0,\infty)\to (0, \infty)$ that dominates Equation \eqref{eq:Assumption:CassLyons1}. Define the operator $\Xi: \cP_K\big( G\Omega_{\alpha}(\bR^{d'})\big) \to \cP_2\big( G\Omega_\alpha(\bR^d)\big)$ by
\begin{equation}
\label{dfn:pro:ContinuityXi}
\Xi\big[ \cL^\rw\big] = \cL^\rx, 
\end{equation}
where $\cL^\rx$ is the unique measure that is a fixed point of Equation \eqref{eq:dfn:FixedPointOperator} so that $\Psi_{\cL^\rw} (\cL^\rx) = \cL^\rx$ . 

Then the operator is well defined and for $\cL^{\rw_1}, \cL^{\rw_2} \in \cP_K(G\Omega_\alpha(\bR^{d'}))$ we have
\begin{equation}
\label{eq:pro:ContinuityXi}
\bW^{(2)}_{\rho_{\alpha-\textrm{H\"ol}; [0,T]} } \Big( \Xi[ \cL^{\rw_1}], \Xi[ \cL^{\rw_2}] \Big) \leq C \bW^{(2)}_{\rho_{\alpha-\textrm{H\"ol}; [0,T]}}\Big( \cL^{\rw_1}, \cL^{\rw_2} \Big)
\end{equation}
with a constant $C=C(\alpha, K, T, d, d')$. 
\end{proposition}

Previously, this result was used to show that the empirical measure obtained by sampling paths of a Brownian motion could be used to obtain a particle system that would converge as the number of particles increased to the solution of a McKean Vlasov Equation. In the remarkable work \cite{deuschel2017enhanced}, the authors study the rate of convergence of these empirical measures to the true law in probability. 

\begin{proof}
Same as proof \cite{CassLyonsEvolving}*{Lemma 4.11}. 
\end{proof}

\subsection{Continuity with respect to the Occupation Measure path}

In \cite{CassLyonsEvolving}*{Theorem 4.9}, the goal was to establish the existence of a contraction operator whose fixed point would be the law of the McKean Vlasov Equation. In fact, computing the specific contraction operator is not simple. Here, we provide a more tangible operator that is (Lipschitz) continuous but not a contraction. 

\begin{proposition}
\label{pro:OccMeasPathLipschitz}
Let $b$ and $\sigma$ satisfy Assumption \ref{Ass:Lipschitz:b+sigma} and let $\Theta_{b, \sigma}$ be the operator from Definition \ref{Def:b.sigma}. 

Then $\Theta_{b, \sigma}$ is Locally Lipschitz continuous in the measure component, that is $\forall \mu, \nu\in \cP_1( G\Omega_\alpha(\bR^d))$ such that 
$$
\int_{G\Omega_\alpha(\bR^d)} \rho_{\alpha-\textrm{H\"ol}}( \rx, \rId) d\mu(\rx), \int_{G\Omega_\alpha(\bR^d)} \rho_{\alpha-\textrm{H\"ol}}( \rx, \rId) d\nu(\rx)<C
$$
and $\forall \xi \in \bR^d$ and $\forall \rw \in G\Omega_\alpha(\bR^{d'})$ such that $\rho_{\alpha-\textrm{H\"ol}; [0,T]}(\rw, 1)<C$, $\exists L_C>0$ such that
\begin{equation}
\label{eq:pro:OccMeasPathLipschitz}
\rho_{\alpha-\textrm{H\"ol}} \Big( \Theta_{b, \sigma}(\mu, \xi, \rw), \Theta_{b, \sigma}(\nu, \xi, \rw)\Big) \leq L_C \bW^{(2)}_{\rho_{\alpha-\textrm{H\"ol}; [0,T]} } \Big( \mu, \nu\Big)
\end{equation}
\end{proposition}

\begin{proof}
Let $p=\tfrac{1}{\alpha}$ and $M=\lfloor \tfrac{1}{\alpha}\rfloor$. Denote the control $\omega(s, t) = \| \rw\|_{p-var; [s, t]}^{p} + \| \gamma^\mu\|_{1-var; [s, t]}+\| \gamma^\nu\|_{1-var; [s, t]}$. Then \cite{CassLyonsEvolving}*{Lemma 4.3} gives
$$
\rho_{p-\omega; [0,T]}\Big( \Theta_{b, \sigma}(\mu, \xi, \rw), \Theta_{b, \sigma}(\nu, \xi, \rw) \Big) \leq C \rho_{1, \omega; [0,T]}(\gamma^\mu, \gamma^\nu) \exp\Big( \rM_{\beta, [0,T]}( \omega ) \Big). 
$$
Indeed, we also have
$$
\| \gamma^\mu_{s, t} - \gamma^\nu_{s, t}\|_{\lip^1( G^{M}(\bR^d))^*} \leq |t-s| \bW^{(2)}_{\rho_{\alpha-\textrm{H\"ol}; [0,t]}}\Big( \mu, \nu\Big). 
$$
By assumption, the Wasserstein distance must be finite across the interval $[0,T]$, so we know the control can be dominated by $\omega(s, t) \lesssim |t-s|$. Thus $\rho_{p-\omega}$ will be equivalent to $\rho_{\alpha-\textrm{H\"ol}}$ and we get
$$
\rho_{\alpha-\textrm{H\"ol}; [0,T]}\Big( \Theta_{b, \sigma}(\mu, \xi, \rw), \Theta_{b, \sigma}(\nu, \xi, \rw) \Big) \leq C \bW^{(2)}_{\rho_{\alpha-\textrm{H\"ol}; [0,t]}}\Big( \mu, \nu\Big) \cdot \exp\Big( \rM_{\beta, [0,T]}(\omega)\Big)
$$
Next, we note that while the constant $C$ is uniform over the choice of $\mu$ and $\nu$, the control $\omega$ is dependent on them and so the Accumulated $\beta$-local p-variation is also dependent on their second moments. 
\end{proof}

With only Proposition \ref{pro:OccMeasPathLipschitz}, one can establish the distance between two paths driven by different occupation measure paths. Next we prove uniform continuity. 

\begin{theorem}
\label{thm:JointContinuityTheta}
Let $b$ and $\sigma$ satisfy Assumption \ref{Ass:Lipschitz:b+sigma} and let $\Theta_{b, \sigma}$ be the operator defined in Definition \ref{Def:b.sigma}. Then the operator $\Theta_{b, \sigma}$ is jointly continuous over $\cP_2\big( G\Omega_\alpha(\bR^d) \big) \times \bR^d \times G\Omega_\alpha(\bR^{d'})$. In particular, 
\begin{align*}
\lim_{(\mu_k, \xi_k, \rw_k) \to (\mu, \xi, \rw)} \Theta_{b, \sigma} ( \mu_k, \xi_k, \rw_k) &= \lim_{\mu_k \to \mu} \lim_{\xi_k \to \xi} \lim_{\rw_k \to \rw} \Theta_{b, \sigma} ( \mu_k, \xi_k, \rw_k) 
\\
&= \Theta_{b, \sigma} ( \mu, \xi, \rw)
\end{align*}
\end{theorem}

\begin{proof}
Let  $\xi, \chi \in \bR^d$ and $p=\tfrac{1}{\alpha}$. For $\rw_1, \rw_2 \in G\Omega_{\alpha}(\bR^{d'})$ and $\mu, \nu \in\cP_2\Big( G\Omega_\alpha(\bR^d) \Big) $, define the control 
$$
\omega(s, t) = \| \rw_1\|_{p-var; [s, t]}^p + \| \rw_2\|_{p-var; [s, t]}^p + \|\gamma^\mu\|_{1-var; [s, t]} + \| \gamma^\nu\|_{1-var; [s, t]}.
$$
We have
\begin{align*}
\rho_{\alpha-\textrm{H\"ol}; [0,T]}&\Big( \Theta_{b, \sigma}( \mu, \xi, \rw_1), \Theta_{b, \sigma}( \nu, \chi, \rw_2) \Big) 
\\
\leq &C \Big( |\xi - \chi| + \rho_{\alpha-\textrm{H\"ol}; [0,T]} ( \rw_1, \rw_2) + \rho_{1-\textrm{H\"ol}; [0,T]} ( \gamma^\mu, \gamma^\nu) \Big) \exp\Big( \rM_{\beta, [0,T]} ( \omega) \Big). 
\end{align*}

Proposition \ref{pro:OccMeasPathLipschitz} shows continuity in measure pointwise for each Geometric rough path $\rw$. Therefore, to prove joint continuity via Moore-Osgood we verify the uniform continuity condition. 

Let $\mu_k, \mu \in \cP_2\big( G\Omega_\alpha(\bR^d) \big)$ and $\bW^{(2)}_{\rho_{\alpha-\textrm{H\"ol}; [0,T]}} (\mu_k, \mu) \to 0$. Then we also have 
$$
\lim_{k\to \infty} \| \gamma^{\mu_k} - \gamma^{\mu}\|_{1-var; [0,T]} =0.
$$
Hence there must exist an $C'\in \bN$ such that 
$$
\sup_{k>C'} \| \gamma^{\mu_k}\|_{1-var; [0,T]} \leq \| \gamma^\mu\|_{1-var;[0,T]} + 1. 
$$
Similarly, by choosing $C'$ large enough
$$
\sup_{k>C'} \| \rw_k\|_{p-var; [0,T]}^p \leq \Big( \| \rw\|_{p-var;[0,T]} + 1\Big)^p. 
$$

Thus 
\begin{align*}
\sup_{k>C'} \rho_{\alpha-\textrm{H\"ol}; [0,T]} &\Big( \Theta_{b, \sigma}( \mu_k, \xi_k, \rw_1), \Theta_{b, \sigma}( \mu_k, \xi_k, \rw_2) \Big) 
\\
\leq& C \rho_{\alpha-\textrm{H\"ol}; [0,T]} ( \rw_1, \rw_2) 
\\
&\cdot \exp\Big( \rM_{\beta, p} (\rw_1) + \rM_{\beta, p} (\rw_2)\Big) \exp\Big( \| \gamma^\mu\|_{1-var; [0,T]} +1 \Big), 
\\
\sup_{k>C'} \rho_{\alpha-\textrm{H\"ol}; [0,T]} &\Big( \Theta_{b, \sigma}( \mu_k, \xi, \rw_k), \Theta_{b, \sigma}( \mu_k, \chi, \rw_k) \Big) 
\\
\leq& C |\xi - \chi| \cdot \exp\Big( \big( \| \rw\|_{p-\textrm{var}; [0,T]}+1\big)^p \Big) \exp\Big( \| \gamma^\mu\|_{1-var; [0,T]} +1 \Big), 
\\
\sup_{k>C'} \rho_{\alpha-\textrm{H\"ol}; [0,T]} &\Big( \Theta_{b, \sigma}( \mu, \xi_k, \rw_k), \Theta_{b, \sigma}( \nu, \xi_k, \rw_k) \Big) 
\\
\leq& C \rho_{1-\textrm{H\"ol}; [0,T]}(\gamma^\mu, \gamma^\nu) 
\\
&\cdot \exp\Big( \big( \| \rw\|_{p-\textrm{var}; [0,T]}+1\big)^p \Big) \exp\Big( \| \gamma^\mu\|_{1-var; [0,T]} +\| \gamma^\nu\|_{1-var; [0,T]} \Big)
\end{align*}
which implies uniform continuity. 
\end{proof}

\section{Support Theorem}
\label{sec:SupportTheo}

Finally,  we state and prove representations of the support of McKean Vlasov Equations in terms of the particle systems associated to the quantizations that we constructed in Section \ref{sec:QuantizationBM}. We introduce a collection of sets of paths that to the best of our knowledge have not previously been described in another work. These sets are all subsets of $C^{\alpha, 0}([0,T]; \bR^{d})$ and are defined solely with respect to the RKHS $\cH$, the H\"older norm $\|\cdot \|_\alpha$ and the coefficients of the Rough Differential Equation \eqref{eq:MVRDE}. 

In order to provide a clear exposition of the construction of the support, we briefly summarise the upcoming subsections: from the previous Section we have obtained a sequence of quantizations $\rQ_n$ for the the law of the enhanced Brownian motion with codebooks $\rC_n$. 
\begin{itemize}
\item For each quantization, we solve the system of interacting ODEs in Section \ref{subsubsec:IntPartsystderQuant} (see Equation \eqref{eq:dfn:InteractingParticle}) by replacing the path of Brownian motion by the associated codebook path and replacing the law of the Brownian motion by the quantization
\item By associating to each of these ODEs the probability weight associated to the codebook element driving the equation, we obtain a finite support measure  in Section \ref{subsubsec:QuantOfMcKeanVlas} (see Equation \eqref{eq:dfn:QuantizationMcKeanVlasov}). We call this the quantization of the McKean Vlasov Equation. This sequence of finite support measures converges to the law of the McKean Vlasov Equation. 
\item In Section \ref{subsubsec:QuantSkelMcKeanVlas}, for fixed $n$, we replace the law of the McKean Vlasov Equation inside the canonical skeleton process by the quantization of the McKean Vlasov Equation (see Definition \ref{dfn:PositiveMeasureSets}). These paths will not generally be contained in the support of the McKean Vlasov Equation. However, a ball of large enough radius will have positive measure (see Lemma \ref{lem:PositiveMeasure} ). 
\item In Section \ref{subsubsec:SuppTheoWithoutKnowledgeSolLaw}, we show that for $n$ chosen large enough, an $\varepsilon$ ball around this collection of paths will be a closed set of measure 1. By taking an intersection of these sets, we show the set of limit points has measure 1 (see Theorem \ref{thm:SupportTheorem1}). 
\item Finally in Section \ref{subsec:RandomInitCond}, we extend our work to the case where the McKean Vlasov Equation has an initial law (see Theorem \ref{thm:SupportTheorem1RandInit}). 
\end{itemize}

\subsection{The Skeleton Process for a McKean Vlasov Equations}

The law of a McKean Vlasov equation is deterministic; it is not dependent on the choice of driving noise. The Occupation Measure path is of bounded variation and does not interact with the noise. Thus when the Occupation Measure path is known, McKean Vlasov Equations can be thought of as Classical Rough Differential Equations with a drift term. Thus, we can define a Skeleton process in the following classical sense:

\begin{definition}
\label{dfn:TrueSkeleton1}
Let $\cL^\rw$ be the law of an enhanced Brownian motion. Let $b$ and $\sigma$ satisfy Assumption \ref{Ass:Lipschitz:b+sigma}. Let $\xi \in \bR^d$. Let $\cL^\rx$ be the unique fixed point of the operator $\Psi_{\cL^\rw}$. Then we define the \emph{True Skeleton Operator} $\rPhi':\cH \times \bR^d \to G\Omega_{\alpha}(\bR^d)$ to be the operator that maps the element of the RKHS to the solution of the ODE
\begin{equation}
\label{eq:TrueSkeleton1}
d\rPhi'(h, \xi)_t = b(\rPhi'(h, \xi)_t) d\gamma_t^{\cL^\rx} + \sigma(\rPhi'(h, \xi)_t) dh_t, \quad \rPhi(h, \xi)_0 = \xi. 
\end{equation}
\end{definition}

It is important to emphasise that the \emph{True Skeleton Operator} \eqref{eq:TrueSkeleton1} is dependent on the measure $\cL^\rx$ and as such it cannot be solved without knowing the law exogenously. The main contribution of this Section is how one navigates around this issue. 

\subsubsection{Interacting Particle system derived from Quantization}
\label{subsubsec:IntPartsystderQuant}

We introduce a system of interacting Ordinary Differential Equations that model the dynamics of the McKean Vlasov Equation. 

\begin{definition}
\label{dfn:InteractingParticle}
Let $\xi\in \bR^d$. Let $\cL \in \cP_c(G\Omega_{\alpha}(\bR^{d'}))$ be a finitely supported measure over the space of Geometric rough paths with the form $\cL = \sum_{j=1}^n \fp_j \delta_{\rw_j}$ where $(\fp_j)_{j=1, ..., n}$ is a Probability vector. For a codebook $\rC:=\{ \rw_j: j=1, ..., n\}$, let $\rw:= \oplus_{j=1}^n \rw_j$ and let $\tilde{\rw}$ be the extension of $\rw$ to $G^M\big(\bR^{d'\times n}\big)$ where $M$ is the largest integer such that $M\alpha<1$. Let $b$ and $\sigma$ satisfy Assumption \ref{Ass:Lipschitz:b+sigma}. Let $B$ and $\Sigma$ be as in Definition \ref{dfn:BandSigma}. 

Then we define the $\cL$-\emph{Interacting Particle System} to be the solution to the Rough Differential Equation
\begin{equation}
\label{eq:dfn:InteractingParticle}
d\rPhi(\cL)_t = B\big(\rPhi(\cL)_t\big) dt + \Sigma\big(\rPhi(\cL)_t \big) d\tilde{\rw}_t, \quad \rPhi(\cL)_0 = \oplus_{j=1}^n \xi \in \bR^{d\times n}
\end{equation}
taking values in $G\Omega_{\alpha}\Big( \bR^{d\times n}\Big)$.
\end{definition}

An important detail about this object is that this is a \emph{finite dimensional} system of Rough Differential Equations. This system of interacting equations can be solved without having to consider any measures. 

The existence and uniqueness of the ODE \eqref{eq:dfn:InteractingParticle} is standard. In particular, by Theorem \ref{thm:ExtensionTheorem1} the solution to Equation \eqref{eq:dfn:InteractingParticle} is independent of the choice of $\tilde{\rw}$ and only on $\rw$. 

\subsubsection{Quantization of the McKean Vlasov}
\label{subsubsec:QuantOfMcKeanVlas}

We use the interacting particle system \eqref{eq:dfn:InteractingParticle} to obtain a law that approximates the law of the McKean Vlasov Equation \eqref{eq:MVRDE}. 

\begin{definition}
\label{dfn:QuantizationMcKeanVlasov}
Let $\cL \in \cP_c(G\Omega_{\alpha}(\bR^{d'}))$ be a finite support measure over the space of geometric rough paths with the form $\cL = \sum_{m=1}^n \fp_m \delta_{\rw^m}$ where $(\fp_m)_{m=1, ..., n}$ is a Probability vector. Let $b$ and $\sigma$ satisfy Assumption \ref{Ass:Lipschitz:b+sigma}. 

Let $\rPhi(\cL)$ be the solution to Equation \eqref{eq:dfn:InteractingParticle}. Let $\pi^{(m)}: G\Omega_\alpha(\bR^{d \times n}) \to G\Omega_\alpha(\bR^d)$ be the quotient operator obtained by extending the projection $\langle \cdot, e_{(\cdot, m)}\rangle$. Then we define the \emph{Law of the $\cL$-Interacting Particle System} to be the finite measure over $G\Omega_\alpha(\bR^d)$
\begin{equation}
\label{eq:dfn:QuantizationMcKeanVlasov}
\cL^{\rPhi(\cL)}: = \sum_{m=1}^n \fp_m \delta_{\pi^{(m)}[\rPhi(\cL)]} .
\end{equation}
\end{definition}

Substituting a quantization of the Brownian motion into an Interacting Particle System and taking its law, we obtain a quantization for the McKean Vlasov Equation. 

\begin{proposition}
\label{pro:IntParticleLaw-TrueLaw}
Let $\cL^\rw$ be the law of enhanced Brownian motion. Let $\cL^\rw \circ \rQ_n^{-1}$ be the sequence of quantizations of the enhanced Brownian motion from Definition \ref{dfn:TheQuantizationRough}. 

Let $\cL^{\rPhi(\cL^\rw\circ \rQ_n^{-1})}$ be the sequence of quantizations for the McKean Vlasov obtained from the sequence of finite support measures  $\cL^\rw \circ \rQ_n^{-1}$. Then $\Xi\Big[ \cL^\rw \circ \rQ_n^{-1}\Big] = \cL^{\rPhi(\cL^\rw\circ \rQ_n^{-1})}$ so that
\begin{equation*}
\label{eq:pro:IntParticleLaw-TrueLaw}
\bW^{(2)}_{\rho_{\alpha-\textrm{H\"ol}; [0,T]}} \Big( \cL^{\rPhi(\cL^\rw\circ \rQ_n^{-1})}, \cL^\rx\Big) \lesssim \Big(\log(n)\Big)^{\alpha-1/2}. 
\end{equation*}
\end{proposition}

\begin{proof}
We have $\Xi\Big[ \cL^\rw \circ \rQ_n^{-1}\Big] = \cL^{\rPhi(\cL^\rw\circ \rQ_n^{-1})}$ and $\Xi\Big[ \cL^\rw\Big] = \cL^\rx$. By Proposition \ref{pro:ContinuityXi}, we have
$$
\bW^{(2)}_{\rho_{\alpha-\textrm{H\"ol}; [0,T]}} \Big( \Xi\big[ \cL^\rw \circ \rQ_n^{-1}\big], \Xi\big[ \cL^\rw\big] \Big) \lesssim \bW^{(2)}_{\rho_{\alpha-\textrm{H\"ol}; [0,T]}} \Big(\cL^\rw\circ \rQ_n^{-1}, \cL^\rw\Big). 
$$
Apply Theorem \ref{thm:ActualRateCon4QuantizationRough} for the rate of convergence. 
\end{proof}

\subsection{The Support of the McKean Vlasov Equation}

The following result immediately holds from the methods laid out in \cite{friz2010multidimensional}*{Chapter 19}. 

\begin{theorem}
\label{thm:UselessSupportTheorem}
Let $\cL^\rw$ be the law of an enhanced Brownian motion. Let $\xi \in \bR^d$. Let $b$ and $\sigma$ satisfy Assumption \ref{Ass:Lipschitz:b+sigma}. Let $\cL^\rx$ be the law of the McKean Vlasov Equation \eqref{eq:MVRDE}. Then the support of $\cL^\rx$ can be characterised with respect to the rough path H\"older metric by
\begin{equation}
\label{eq:thm:UselessSupportTheorem}
\supp( \cL^\rx) = \overline{ \Big\{ \rPhi'(h, \xi): h\in \cH \Big\}}^{\rho_{\alpha-\textrm{H\"ol}; [0,T]}} 
\end{equation}
where $\rPhi'$ is the True Skeleton operator from Definition \ref{dfn:TrueSkeleton1}. 
\end{theorem}

This is not a meaningful result as the True Skeleton Operator includes a priori knowledge of the law of the McKean Vlasov Equation. This measure can be proved to exist, but constructing it is another matter. We overcome this issue via functional quantization. 

\subsubsection{Quantized Skeleton of McKean Vlasov}
\label{subsubsec:QuantSkelMcKeanVlas}

We use the quantized McKean Vlasov to construct a Skeleton process that approximates the True Skeleton Process. 

\begin{definition}
\label{dfn:PositiveMeasureSets}
Let $\cL^\rw$ be the law of an enhanced Brownian Motion.  Let $\rQ_n$ be the sequence of quantizations of $\cL^\rw$ constructed in Definition \ref{dfn:TheQuantizationRough}. Let $\xi\in \bR^d$ and let $h\in \cH$ and denote $\rh=S_2[h]$. Let $b$ and $\sigma$ satisfy Assumption \ref{Ass:Lipschitz:b+sigma}.

Fix $\varepsilon>0$ and choose $n\in \bN$ such that
\begin{equation}
\label{eq:dfn:PositiveMeasureSets1}
\rho_{\alpha-\textrm{H\"ol}; [0,T]}\Big( \Theta_{b, \sigma}(\cL^{\rPhi(\cL^\rw \circ \rQ_n^{-1})}, \xi, \rh) , \Theta_{b, \sigma}(\cL^\rx, \xi, \rh) \Big)
\leq \varepsilon,
\end{equation}
and we define the sets $A_\varepsilon(\rh)$ as
\begin{equation}
\label{eq:dfn:PositiveMeasureSets2}
A_\varepsilon(\rh):=\Big\{ \ry\in G\Omega_{\alpha}(\bR^d): \rho_{\alpha-\textrm{H\"ol}; [0,T]} \Big(\ry, \Theta_{b,\sigma}(\cL^{\rPhi(\cL^\rw \circ \rQ_n^{-1})}, \xi, \rh) \Big)< \varepsilon \Big\}. 
\end{equation}
\end{definition}

We emphasise that the choice of $n$ will not be uniform over all choices of $h\in \cH$. Also note that $\rPhi'(h, \xi) = \Theta_{b, \sigma}(\cL^\rx, \xi, \rh)$. The first goal is to show that each of these sets contains an element of the $\supp(\cL^\rw)$, regardless of $\varepsilon$. 

\begin{lemma}
\label{lem:PositiveMeasure}
Let $h\in \cH$ and $\rh=S_2[h]$. Then $\forall \varepsilon>0$, the open sets $A_\varepsilon(\rh)$ of Definition \ref{dfn:PositiveMeasureSets} have positive measure with respect to $\cL^\rx$, 
$$
\cL^\rw\Big[ A_\varepsilon(\rh)\Big]>0. 
$$
\end{lemma}

\begin{proof}
The condition for $A_\varepsilon(\rh)$ in Equation \eqref{eq:dfn:PositiveMeasureSets1} is the key. It ensures that for any choice of $\varepsilon>0$, we have $\rPhi'(h, \xi) \in A_\varepsilon(\rh)$. By Theorem \ref{thm:UselessSupportTheorem}, we have that any open set $B\subseteq G\Omega_{\alpha}(\bR^d)$ containing a path $\rPhi'(h, \xi)$ and for any choice of $h\in \cH$, we have 
$$
\cL^\rx[B] >0. 
$$
\end{proof}

\subsubsection{The Support Theorem without knowledge of the solution law}
\label{subsubsec:SuppTheoWithoutKnowledgeSolLaw}

We now formulate our statement of the support theorem of McKean Vlasov Equations: 

\begin{theorem}
\label{thm:SupportTheorem1}
Let $\cL^\rw$ be the law of an enhanced Brownian motion. Let $\rQ_n$ be the sequence of quantizations obtained in Definition \ref{dfn:TheQuantizationRough}. Let $\cL^{\rPhi(\cL^\rw \circ \rQ_n^{-1})}$ be the law of the Interacting Particle System driven by the quantization constructed in Definition \ref{dfn:QuantizationMcKeanVlasov}. Let $\xi\in \bR^d$. Suppose that $b$ and $\sigma$ satisfy Assumption \ref{Ass:Lipschitz:b+sigma}. Then the law of the solution to the McKean Vlasov Equation \eqref{eq:MVRDE} satisfies 
\begin{equation}
\label{eq:thm:SupportTheorem1}
\supp(\cL^\rx) = \bigcap_{m=1}^\infty \overline{ \bigcup_{n\geq m} \Big\{ \Theta_{b, \sigma}(\cL^{\rPhi(\cL^\rw \circ \rQ_n^{-1})}, \xi, \rh): h\in \cH, \rh = S_2(h) \Big\} }^{\rho_{\alpha-\textrm{H\"ol}; [0,T]}}. 
\end{equation}
\end{theorem}

We emphasise that this expression of the support is only dependent on:
\begin{itemize}
\item The RKHS of Brownian motion $\cH$ and the initial condition $\xi\in \bR^d$
\item The coefficients $b$ and $\sigma$
\item The sequence of Systems of Interacting Particles $\rPhi(\cL^\rw \circ \rQ_n^{-1})$ which is in turn dependent on
\begin{itemize}
\item The coefficients $b$ and $\sigma$
\item The sequence of quantizations $\rQ_n$ which are only dependent on $\cH$ and $\| \cdot \|_\alpha$. 
\end{itemize}
\end{itemize}

We have not solved the law of the McKean Vlasov Equation or the Occupation measure path at any point of this approach. 

\begin{proof}
For the simplicity of the proof, we rely on Theorem \ref{thm:UselessSupportTheorem} for an expression of $\supp(\cL^\rx)$. By Proposition \ref{pro:IntParticleLaw-TrueLaw}, we have that the law of the Interacting Particle System converges to the law of the McKean Vlasov Equation as $n\to \infty$. Fix $h\in \cH$ and $m\in \bN$. Then $\forall l\geq m$
$$
\Theta_{b, \sigma}( \cL^{\rPhi(\cL^\rw \circ \rQ_l^{-1})}, \xi, \rh) \in \overline{ \bigcup_{n\geq m} \Big\{ \Theta_{b, \sigma}(\cL^{\rPhi(\cL^\rw \circ \rQ_{n}^{-1})}, \xi, \rh): h\in \cH, \rh = S_2(h) \Big\} }^{\rho_{\alpha-\textrm{H\"ol}; [0,T]}}. 
$$
Since this is closed, we have that the limit of these paths is also contained so
\begin{equation}
\label{eq:thm:SupportTheorem1.1}
\rPhi'(h, \xi) = \Theta_{b, \sigma}( \cL^{\rx}, \xi, \rh) \in \overline{ \bigcup_{n\geq m} \Big\{ \Theta_{b, \sigma}(\cL^{\rPhi(\cL^\rw \circ \rQ_n^{-1})}, \xi, \rh): h\in \cH, \rh = S_2(h) \Big\} }^{\rho_{\alpha-\textrm{H\"ol}; [0,T]}}. 
\end{equation}
Finally, Equation \eqref{eq:thm:SupportTheorem1.1} holds for any choice of $m\in \bN$, so it must be contained in the intersection over all $m$. This was true for any choice of $h\in \cH$, so it is also true for all $h\in \cH$. Thus
$$
\Big\{ \rPhi'(h, \xi): h\in \cH\Big\} \subset \bigcap_{m=1}^\infty \overline{ \bigcup_{n\geq m} \Big\{ \Theta_{b, \sigma}(\cL^{\rPhi(\cL^\rw \circ \rQ_n^{-1})}, \xi, \rh): h\in \cH \Big\} }^{\rho_{\alpha-\textrm{H\"ol}; [0,T]}}. 
$$
Finally, as the right hand side is closed, we can take a closure on the left hand side to achieve the first implication. 

Now we show the reverse implication. Suppose $\ry\in G\Omega_{\alpha}(\bR^d)$ such that
$$
\ry \in \bigcap_{m=1}^\infty \overline{\bigcup_{n\geq m} \Big\{ \Theta_{b, \sigma}(\cL^{\rPhi(\cL^\rw \circ \rQ_n^{-1})}, \xi, \rh): h\in \cH \Big\} }^{\rho_{\alpha-\textrm{H\"ol}; [0,T]}}. 
$$
Then there must exist a subsequence $n_k$ and a sequence of $h_k \in \cH$ such that
$$
\lim_{k\to \infty} \rho_{\alpha-\textrm{H\"ol}; [0,T]}\Big( \Theta_{b, \sigma}( \cL^{\rPhi(\cL^\rw \circ \rQ_{n_k}^{-1})}, \xi, \rh_k),  \ry\Big) = 0. 
$$
Further, we know the sequence satisfies $\lim_{k\to \infty} n_k = \infty$, since $\ry$ is in the intersect over all $m\in \bN$. Thus the weak limit of $\cL^{\rPhi(\cL^\rw \circ \rQ_{n_k}^{-1})}$ must just be $\cL^\rx$ as $k\to \infty$. 

By Theorem \ref{thm:JointContinuityTheta}, we have Joint Continuity of $\Theta_{b, \sigma}$. Therefore, taking the limit in the measure variable first, we get
$$
\lim_{k\to \infty} \rho_{\alpha-\textrm{H\"ol}; [0,T]}\Big( \Theta_{b, \sigma}( \cL^\rx, \xi, \rh_k),  \ry\Big) = \lim_{k\to \infty} \rho_{\alpha-\textrm{H\"ol}; [0,T]}\Big( \rPhi'(h_k, \xi),  \ry\Big)= 0,
$$
which just means that $\ry\in \overline{\{ \rPhi'(\rh, \xi): h\in \cH\}}^{\rho_{\alpha-\textrm{H\"ol};[0,T]}}$. 
\end{proof}

\subsection{Random Initial Conditions}
\label{subsec:RandomInitCond}

An apparent limitation of the previous Section is that we restrict ourselves to McKean Vlasov Equations with constant initial conditions. However, there is an easy extension to the case where the initial condition is random. 

We introduce a Theorem first proved in \cite{caballero1997composition} that allows for the consideration of random initial conditions. 

\begin{theorem}[\cite{caballero1997composition}]
Let $F:\Omega \times \bR^d \to E$ be a random variable taking values in a Banach space $E$ such that $x\mapsto F(\omega, x)$ is continuous for each $\omega$. Suppose that $G:\cH \times \bR^d \to E$ is a uniform skeleton of $F$. Suppose that $\zeta$ is an $d$-dimensional random variable with skeleton $\phi$. Then $\tilde{G}(h):= G(h, \phi(h))$ is a skeleton of $\tilde{F}(\omega):=F(\omega, \zeta(\omega))$. 
\end{theorem}

We now turn to the McKean Vlasov Equation
\begin{equation}
\label{eq:MVRDERandInit}
dX_t = \sigma(X_t)d\rw_t + b(X_t) d\gamma^{\cL_t^\rx}, \quad X_0 \sim \xi\in \cP_{r}( \bR^d)
\end{equation}
where $r>1$. 

Following in the footsteps of Definition \ref{dfn:TheQuantizationRough}, we construct a quantization for the law $\xi \times \cL^\rw$ over $\bR^d \times G\Omega_{\alpha}(\bR^{d'})$. 

\begin{definition}
\label{dfn:TheQuantizationRoughRandInit}
Let $r>1$. Let $\cL^W$ be the law of a Brownian motion over $C^{\alpha, 0}([0,T]; \bR^{d'})$. Let $\xi \in \cP_{r}(\bR^d)$. Let $m, n\in \bN$. 
\begin{enumerate}
\item By Theorem \ref{theorem:ExistenceStationaryQuant}, there exists a codebook $\fC_m^{(1)} \subset \bR^d$ that is an $m$-stationary set with Voronoi partition $\fS_m^{(1)}$. Let $\rC_n^{(2)}$ be the $n$ element codebook constructed in Definition \ref{dfn:TheQuantizationRough} with partition $\rS_n^{(2)}$. 
\item Let $\rC_{m,n}:= \fC_m^{(1)} \times \rC_n^{(2)}$ be a sequence of codebooks over $\bR^d \times G\Omega_\alpha(\bR^{d'})$ and let $\rS_{m,n}:=\fS_m^{(1)} \times \rS_n^{(2)}$ be a partition of $\bR^d \times G\Omega_\alpha(\bR^{d'})$. Let $\rQ_{m, n}$ be the Quantization with codebook $\rC_{m,n}$ and partition $\rS_{m,n}$. Then $| \rC_{m, n}| = m \cdot n$. 
\item By combining Equation \eqref{pro:GrafLuschgyQuantizationRate} and Theorem \ref{thm:ActualRateCon4QuantizationRough}, the rate of convergence is 
\begin{align}
\tfrac{1}{m^{1/d}} + \Big( \log(n)\Big)^{\alpha-1/2}\approx&  \Bigg( \int_{\bR^d\times G\Omega_\alpha(\bR^{d'})}  d_{| \cdot | \times \rho_{\alpha\textrm{-H\"ol}}} \Big((x, \ry),\rQ_{m,n}(x, \ry)\Big)^r  d[\xi\times\cL^\rw](x, \ry) \Bigg)^{1/r} . 
\end{align}
\item By choosing $m \approx [ \log(n)]^{(1/2-\alpha)d}$ and rescaling, we obtain the sequence of quantizations
\begin{align}
\nonumber
&
\Bigg( \int_{\bR^d\times G\Omega_\alpha(\bR^{d'})}  d_{| \cdot | \times \rho_{\alpha\textrm{-H\"ol}}} \Big((x, \ry),\rQ_{n}(x, \ry)\Big)^r  d[\xi\times\cL^\rw](x, \ry) \Bigg)^{1/r}
\\
\label{eq:dfn:TheQuantizationRoughRandInit}
&\quad \quad \approx 
\Bigg[ \log\Big( \tfrac{n}{[(1/2-\alpha)d]^{(1/2-\alpha)d}} \Big) - \log\Bigg( \cW\Bigg(\frac{n^{\tfrac{1}{(1/2-\alpha)d}}}{(1/2-\alpha)d} \Bigg)^{(1/2-\alpha)d}\Bigg) \Bigg]^{(\alpha-1/2)d}
\end{align}
where, as in Proposition \ref{pro:ActualRateCon4Quantization}, $\cW$ is the Lambert W function. 
\end{enumerate}
\end{definition}

Next, following Definition \ref{dfn:QuantizationMcKeanVlasov}, we define a new interacting particle system.
\begin{definition}
\label{dfn:QuantizationMcKeanVlasovRandInit}
Let $\cL \in \cP_c( \bR^d \times G\Omega_{\alpha}(\bR^{d'}))$ be a finite support measure of the form $\cL = \sum_{j=1}^n \fp_j \delta_{(x_j,\rw_j)}$ where $(\fp_j)_{j=1, ..., n}$ is a probability vector. For codebook $\rC:=\{ (x_j,\rw_j): j=1, ..., n\}$, let $\rw:= \oplus_{j=1}^n \rw_j$ and $X=\oplus_{j=1}^n x_j \in \bR^{d\times n}$. Let $\tilde{\rw}$ be the lift of the path $\rw$ to a rough path. Let $b$ and $\sigma$ satisfy Assumption \ref{Ass:Lipschitz:b+sigma}. Let $B$ and $\Sigma$ be as in Definition \ref{dfn:BandSigma}. 

Then we define the $\cL$ \emph{Interacting Particle System} with random initial condition to be the solution to the Rough Differential Equation
\begin{equation}
\label{eq:dfn:InteractingParticleRandInit}
d\rPhi(\cL)_t = B\big(\rPhi(\cL)_t\big) dt + \Sigma\big(\rPhi(\cL)_t \big) d\tilde{\rw}_t, \quad \rPhi(\cL)_0 = X. 
\end{equation}

We also define the law of the $\cL$ Interacting Particle system in $\cP_c\Big( G\Omega_\alpha(\bR^d)\Big)$ to be
\begin{equation*}
\cL^{\rPhi(\cL)}: = \sum_{m=1}^n \fp_m \delta_{\pi^{(m)}[ \rPhi(\cL)]} . 
\end{equation*}
\end{definition}

As with Theorem \ref{thm:ExtensionTheorem1}, the paths of this law are dependent only on $\rw$ and not of the lift of $\tilde{\rw}$. In this definition we do not limit ourselves to the case where many of the $x_j$ values are repeated. 
We use the quantization of the measure $\xi \times \cL^\rw $ constructed in Definition \ref{dfn:TheQuantizationRoughRandInit} to solve the law of an Interacting Particle system that approximates the true law of the McKean Vlasov Equation

\begin{proposition}
\label{pro:IntParticleLaw-TrueLawRandInit}
Let $\cL^\rw$ be the law of the enhanced Brownian motion. Let $[\xi \times \cL^\rw] \circ \rQ_{n}^{-1}$ be the sequence of quantizations of the enhanced Brownian motion from Definition \ref{dfn:TheQuantizationRoughRandInit}. 

Let $\cL^{\rPhi([\xi\times \cL^\rw]\circ \rQ_{n}^{-1})}$ be the sequence of quantizations for the McKean Vlasov obtained from the sequence of finite support measures  $[\xi \times \cL^\rw] \circ \rQ_{n}^{-1}$. 

Then $\Xi\Big[ [\xi \times \cL^\rw] \circ \rQ_{n}^{-1}\Big] = \cL^{\rPhi([\xi\times \cL^\rw]\circ \rQ_{n}^{-1})}$ so that
\begin{align*}
\bW^{(1)}_{\rho_{\alpha-\textrm{H\"ol}}}& \Big( \cL^{\rPhi([\xi\times \cL^\rw]\circ \rQ_n^{-1})}, [\xi \times \cL^\rx]\Big) 
\\
&\lesssim 
\Bigg[ \log\Big( \tfrac{n}{[(1/2-\alpha)d]^{(1/2-\alpha)d}} \Big) - \log\Bigg( \cW\Bigg(\frac{n^{\tfrac{1}{(1/2-\alpha)d}}}{(1/2-\alpha)d} \Bigg)^{(1/2-\alpha)d}\Bigg) \Bigg]^{(\alpha-1/2)d} . 
\end{align*}
\end{proposition}

\begin{proof}
Same method as Proposition \ref{pro:IntParticleLaw-TrueLaw} with Equation \eqref{eq:dfn:TheQuantizationRoughRandInit}. 
\end{proof}

\subsubsection{Statement for the Support}

Using classical tools, we combine the results of Theorem \ref{thm:UselessSupportTheorem} with \cite{caballero1997composition} for this next Theorem:
%

\begin{theorem}
\label{thm:SupportTheorem1RandInit}
Let $r>1$. Let $\xi\in \cP_{r}(\bR^d)$. Let $\cL^\rw$ be the law of an enhanced Brownian motion. Let $\rQ_{n}$ be the sequence of quantizations obtained in Definition \ref{dfn:TheQuantizationRoughRandInit}. Let $\cL^{\rPhi([\xi\times\cL^\rw] \circ \rQ_{n}^{-1})}$ be the law of the Interacting Particle System driven by the quantization constructed in Definition \ref{dfn:QuantizationMcKeanVlasovRandInit}. Suppose that $b$ and $\sigma$ satisfy Assumption \ref{Ass:Lipschitz:b+sigma}. Then the law of the solution to the McKean Vlasov Equation \eqref{eq:MVRDE} satisfies 
\begin{equation}
\label{eq:thm:SupportTheorem1RandInit}
\supp(\cL^\rx) = \bigcap_{m=1}^\infty \overline{ \bigcup_{n\geq m} \Big\{ \Theta_{b, \sigma}(\cL^{\rPhi([\xi\times\cL^\rw] \circ \rQ_n^{-1})}, x, \rh): h\in \cH, x\in \supp(\xi) \Big\} }^{\rho_{\alpha-\textrm{H\"ol}; [0,T]}}. 
\end{equation}
\end{theorem}

\begin{proof}
See proof of Theorem \ref{thm:SupportTheorem1} with Proposition \ref{pro:IntParticleLaw-TrueLawRandInit} and Theorem \ref{thm:JointContinuityTheta}. 
\end{proof}

\section*{Acknowledgements}
In no particular order, the authors thank Professor Sandy Davie, Professor Fran\c cois Delarue, Professor Peter Friz and Dr Mario Maurelli for the helpful discussions. We also thank the participants of the December 2018 10th Oxford-Berlin \emph{Young Researchers Meeting on Applied Stochastic Analysis} for their feedback on the occasion of the first presentation of this work. 


\appendix

\section{Rough Path Primer}
\label{Appendix:Primer}

\subsection{Algebraic Material}

Let $\cA$ be a finite alphabet, let $V$ be the associated vector space and denote $T(V)=\oplus_{n=0}^\infty V^{\otimes n}$ be the vector space of free monoids generated by $\cA$ with the shuffle product $\shuffle$. Let $\Delta: T(V) \to T(V) \otimes T(V)$ be the deconcatenation coproduct. Thus $(T(V), \shuffle, \Delta)$ is a commutative unital Hopf algebra with an antipode and canonical grading. 

The characters (also known as Group-like elements) of $T(V)$ to be the elements $g\in G(V)$ such that $\forall u, v\in T(V)$
$$
\langle g, u\shuffle v\rangle = \langle g, u\rangle \langle g, v\rangle. 
$$
$G(V)$ forms a Lie group with Lie Algebra $P(V)$. The diffeomorphic exponential map $\exp_\boxtimes :P(V) \to G(V)$ and its inverse the logarithm map $\log_\boxtimes : G(V) \to P(V)$ defined for $g\in G(V)$ and $h\in P(V)$ by
$$
\exp_\boxtimes (h) = \sum_{i=0}^\infty \frac{h^{\boxtimes i}}{i!} , \quad \log_\boxtimes (g) = \sum_{i=1}^\infty (-1)^{i-1} \frac{(g-\rId)^{\boxtimes i}}{i}. 
$$
Finally, we define $T^M(V)$ to be the quotient space obtained from $T(V)$ by quotienting against the ideal $\oplus _{n=M+1}^\infty$. The Lie algebra $P^M(V)$ is graded, so can be expressed as
$$
P^M(V) = \oplus_{i=1}^M V_i
$$
where $V_{i+1} = [V, V_i]_\boxtimes$ and $V_1 = V$. We define the dilation on $P^M(V)$ to be the linear map $\delta_t:P^M(V) \to P^M(V)$ such that
$$
\delta_t[ h_1 + ... + h_M] = t h_1 + ... + t^M h_M. 
$$
Similarly, the dilation can be extended to the Lie Group $\delta_t: G^M(V) \to G^M(V)$ for $g = \exp_\boxtimes(h_1 + ... + h_M)$ by
$$
\delta_t g = \exp_\boxtimes( t h_1 + ... + t^M h_M). 
$$

A homogeneous norm on a Carnot group is a function $\| \cdot \|_G: G \to \bR^+$ such that for any $g\in G$,  $\|g\|_G = 0$ if and only if $g=\rId$ the unit of $\boxtimes$ and $\| \delta_t g\|_G = |t| \cdot \| g\|_G$. 

As the Lie Algebra $P^M(V)$ is finite dimensional, all homogeneous norms on $G^M$ are equivalent. By considering the collection of homogeneous norms, one can induce a left invariant metric over $G^M(\bR^{d'})$. This is traditionally called the Carnot-Carath\'eodory metric which we denote by $d_{cc}$. Further, the Carnot-Carath\'eodory norm satisfies the additional properties for any $g, g_1, g_2\in G^m(\bR^{d'})$ $\| g_1 \boxtimes g_2\|_{cc} \leq \|g_1\|_{cc} + \|g_2\|_{cc}$ and $\| g\|_{cc} = \| g^{-1} \|_{cc}$. 

Let $\cA^M$ be all the words generated by the Alphabet $\cA$ such that $|A|\leq M$. One example of a homogeneous norm that we work with is
\begin{equation}
\label{eq:HomoNorm}
\| g \|_{G^M} = \sum_{A \in \cA_M} | \langle \log_\boxtimes (g), e_A\rangle |^{1/|A|}. 
\end{equation}

\subsection{Rough Paths}

\begin{definition}
Let $V$ be a vector space. For a path $x\in C^{1-var}([0,T]; V)$, the iterated integrals of $x$ are canonically defined using Young integration. The collection of iterated integrals of the path $x$ is called the truncated Signature of $x$ and is defined as
$$
S_M(x)_{s, t}:= \rId + \sum_{n=1}^M \int_{s\leq u_1\leq ... \leq u_n\leq t} dx_{u_1} \otimes ... \otimes dx_{u_n} \in T^M(V) = \bigoplus_{n=0}^M V^{\otimes n}. 
$$
\end{definition}

It is well known that $S_M(x)$ takes values in $G^M(V)$. 

\begin{definition}
For $\alpha\in (0, 1)$ and let $M$ be the largest integer such that $M\alpha<1$. A path $\rx:[0,T] \to G^M(V)$ is called an $\alpha$-H\"older continuous geometric rough paths if
\begin{align}
\nonumber
\langle \rx_{s, t}, e_A\rangle \langle \rx_{s, t}, e_B\rangle = \langle \rx_{s, t}, e_{A} \shuffle e_B \rangle&, 
\quad
\langle \rx_{s, u}, e_A\rangle = \langle \rx_{s, t} \boxtimes \rx_{t, u}, \Delta[e_A] \rangle
\\
\mbox{and} \quad
\sup_{A\in \cA_M} \sup_{s, t\in [0,T]}& \frac{\langle \rx_{s, t}, e_A\rangle }{|t-s|^{\alpha |A|}}< \infty. 
\end{align}

\end{definition}

\begin{definition}
Denote $p=\tfrac{1}{\alpha}$. We define the $\alpha$-H\"older rough path metric
\begin{equation}
\label{eq:HolderDef}
d_\alpha( \rx, \ry) = \| \rx^{-1} \boxtimes \ry\|_\alpha = \sup_{s, t\in[0,T]} \frac{\Big\| \rx_{s,t}^{-1}\boxtimes \ry_{s,t} \Big\|_{cc} }{|t-s|^\alpha}. 
\end{equation}
By quotienting with respect to $\rx_0$, one can make this a norm. We use the convention that $\| \rx \|_{p-var; [0,T]} = \| \rId^{-1}\boxtimes \rx \|_{p-var; [0,T]}$ and $\| \rx \|_{\alpha} = \| \rId^{-1}\boxtimes \rx \|_{\alpha}$. We denote the metric space of $\alpha$-H\"older continuous geometric rough paths to be $G\Omega_\alpha(\bR^d)$. 

Similarly, we define the homogeneous $p$-variation metric $d_{p-var}$ by
\begin{equation}
\label{eq:pVarDef}
d_{p-var; [0,T]}(\rx, \ry):=\| \rx^{-1} \boxtimes \ry\|_{p-var; [0,T]}:=\bigg( \sup_{D=(t_i)} \sum_{i: t_i\in D} \Big\| \rx_{t_i, t_{i+1}}^{-1}\boxtimes \ry_{t_i, t_{i+1}} \Big\|_{cc}^p \bigg)^{\tfrac{1}{p}}. 
\end{equation}

\end{definition}


When studying rough paths, one can either work with $p$-variation or $\alpha$-H\"older norms. For the most part, authors choose one and stick with it for the entirety of their work. While $p$-variation is slightly more general, $\alpha$-H\"older allows for a wavelet representation in the Banach space which is more favourable for this work. 

It is important to understand that for this paper, we work with both norms. The H\"older norm, being more restrictive, is assumed to be the bound on regularity. However, we are required to work with the $p$-variation in order to establish an integrability condition. 

\begin{definition}
\label{def:control1}
Let $\Delta_T = \{(s, t): 0\leq s\leq t \leq T\}$ denote the two-dimensional simplex. The map $\omega:\Delta_T \to \bR^+$ is a Control if it is a continuous, non negative, super-additive function which vanishes on the diagonal. 
\end{definition}

\begin{example}
Suppose that $\rx$ is a geometric rough path with finite $p$-variation, so that Equation \eqref{eq:pVarDef} is finite. Then $\omega_{\rx, p}(s, t):= \| \rx \|_{p-var; [s, t]}^p$ is a control. 
\end{example}

The Carnot-Carath\'eodory metric as already described takes its structure from the Group $G^M(\bR^{d'})$ and so is homogeneous with respect to the group dilation $\delta_\lambda$. However, there is another metric that takes its structure from the vector space $T^M(\bR^{d'})$. 

For two elements $g_1, g_2 \in T^M(\bR^{d'})$ and $i\in\{1, ..., M\}$ we have the collection of pseudo-metrics
\begin{equation}
\label{eq:pseudometric?}
\rho_i(g_1, g_2) = \sum_{\substack{A\in \cA_M\\ |A|=i}} \Big| \langle g_1, e_A\rangle - \langle g_2, e_A\rangle \Big|. 
\end{equation}
We also have the inhomogeneous Tensor metric
$$
\rho(g_1, g_2) = \max_{i=1, ..., M} \rho_i(g_1, g_2). 
$$

\begin{definition}
\label{dfn:InhomogeneousMetric}
Let $p=\tfrac{1}{\alpha}>2$. For a fixed control $\omega$, we define the inhomogeneous $\omega$-modulus metric to be
\begin{equation}
\label{eq:dfn:InhomogeneousOmega-Var}
\rho_{p-\omega; [0,T]}(\rx, \ry):= |\rx_0 - \ry_0|_{T^{\lfloor p\rfloor}(\bR^{d'})} 
+ \max_{i=1, ..., \lfloor p\rfloor} \sup_{s, t\in[0,T]} \frac{ \rho_i( \rx_{s, t}, \ry_{s, t}) }{\omega(s, t)^{i/p}}. 
\end{equation}

When we additionally have that $\omega(s, t)\leq C|t-s|$ where $C$ is a constant independent of $s, t$, we also have the inhomogeneous $\alpha$-H\"older metric to be
\begin{equation}
\label{eq:dfn:InhomogeneousHol}
\rho_{\alpha-\textrm{H\"ol}; [0,T]}(\rx, \ry):= |\rx_0 - \ry_0|_{T^{\lfloor p\rfloor}(\bR^{d'})} 
+ \max_{i=1, ..., \lfloor p\rfloor} \sup_{s, t\in[0,T]} \frac{ \rho_i( \rx_{s, t}, \ry_{s, t}) }{|t-s|^{\alpha i}}. 
\end{equation}
\end{definition}

The inhomogeneous rough path metrics satisfy the simple relation
\begin{equation}
\label{eq:relationforInhomoMetrics}
\rho_{p-var; [0,T]}(\rx, \ry) \leq \Big( 1\vee \max_{i=1, ..., \lfloor p\rfloor} \omega(0,T)^{i/p}\Big) \rho_{p-\omega; [0, T]}(\rx, \ry)
\end{equation}
by simple manipulation of the standard relation between $p$-variation and $\tfrac{1}{p}$-H\"older regularity, see \cite{friz2010multidimensional}. 

\begin{definition}
\label{definition:SteinNotation}
Let $E$ and $F$ be normed spaces. A map $f:E\to F$ is called $\gamma$-Lipschitz (in the sense of Stein) if $f$ is $\lfloor \gamma\rfloor$ continuously differentiable (in the sense of Fr\'echet) and such that there exists a constant $M<\infty$ such that the supremum norm of the $k^{th}$ derivative for $k=1, ..., \lfloor\gamma\rfloor$ and the $\{\gamma\}$-H\"older norm of its $\lfloor \gamma \rfloor^{th}$ derivative are bounded by $M$. The smallest $M\geq 0$ satisfying this condition is the $\gamma$-Lipschitz norm of $f$, denoted $\|f\|_{\lip^\gamma}$. The space of all such functions is denoted $\lip^\gamma(E, F)$. 

We also emphasise the distinction between $\lip_*^1(E, F)$, the space of functions $f:E\to F$ that are Lipschitz. 
\end{definition}


\begin{theorem}[\cite{lyons2007extension}]
Let $V=\bigoplus V^j$ be a vector space. 

Let $\alpha<1/2$ such that $\tfrac{1}{\alpha} \notin \bN$ and $M=\lfloor\tfrac{1}{\alpha}\rfloor$. Suppose that $\rx_t^j$ are $\alpha$-H\"older continuous paths taking values in $G^M(V^j)$. Then $\bigoplus_j \rx_t^j$ can be thought of as an $\alpha$-H\"older continuous path taking values in $\bigoplus_j G^M(V^j)$ and there exists an extension $\rx_t$ taking values in $G^M(V)$ that is $\alpha$-H\"older continuous with respect to the Carnot norm on $G^M(V)$. 
\end{theorem}

\subsection{Controlled Rough Path}

A controlled rough path, first introduced in \cite{gubinelli2004controlling}, provides a path that is known to be adequately regular enough to be integrable with respect to a rough path. 

Let $V$ and $U$ be vector spaces and denote by $L(V, U)$ the space of Linear operators from $V$ to $U$. We define $T(V^*, U):= \bigoplus_{n=0}^\infty L\big((V^*)^{\otimes n}, U\big)$ and use the convention that $L\big( (V^*)^{\otimes 0}, U) = U$. As earlier, we are interested in the case where $V=\bR^{d'}$ and $U=\bR^d$. 

Given an element $\rx\in T(V^*)$ and $\bY\in T(V^*, U)$, we naturally obtain $\bY \rx \in U$. Also, in practice we work in the truncated tensor algebra $T^M(V^*, U):= \bigoplus_{n=0}^M L\big((V^*)^{\otimes n}, U\big)$ obtained by quotienting with respect to the ideal $\bigoplus_{n=M+1}^\infty L\big((V^*)^{\otimes n}, U\big)$. 

\begin{definition}
\label{dfn:ControlledRoughPath}
Let $\alpha\in(0, 1/2)$, let $M$ be the smallest integer such that $M\alpha < 1$ and let $\rx\in G\Omega_\alpha(V)$. Let $\cA_v$ be the alphabet of $V$. 

A $\rx$-controlled rough path $\bY:[0,T] \to T^{M-1}(V, U)$ and a remainder term $R: \Delta_T \to T^{M-1}(V, U)$ is any path such that for any word $A$ of the alphabet for $\cA_v$
$$
\langle \bY_{t}, e_A\rangle - \langle \bY_s, \rx_{s, t} \boxtimes e_A \rangle = \langle R_{s, t}, e_A\rangle,
$$
where
$$
\label{eq:dfn:ControlledRoughPath2}
\sup_{t,s\in [0,T]} \frac{|\langle R_{s, t}, e_A\rangle|}{|t-s|^{(M-|A|)\alpha}} < \infty. 
$$
The space of $\rx$-controlled rough paths, denoted $\cD_\rx^{M\alpha}( [0,T]; U)$ is the vector space of all $\rx$-Controlled paths with the norm
$$
\| \bY\|_{\rx, M\alpha} = \sum_{A\in \cA_M\backslash \{\varepsilon\}} \Big\| \langle \bY, e_A\rangle \Big\|_{|A|\alpha-\textrm{H\"ol}; [0,T]} \quad + \Big\| \langle R, e_\varepsilon\rangle \Big\|_{M\alpha-\textrm{H\"ol}; [0,T]}. 
$$
\end{definition}
Given an $\rx$-controlled rough path $\bY$ taking values on $L(V, U)$, we define the integral
$$
\int_0^T Y_t d\rx_t = \lim_{|D| \to 0} \sum_{i: t_i\in D} \Big\langle \bY_{t_i} , \rx_{t_i, t_{i+1}} \Big\rangle_{T^{M-1}\big(V, L(V, U)\big), T^{M}\big(V\big)}
$$
taking values in $U$. 

\begin{definition}
The Shuffle product over $T(V)$ can be represented as two Left and Right Half-shuffle products $e_A \shuffle e_B = e_A \prec e_B + e_A \succ e_B$ that satisfy the identities
\begin{align*}
(e_A \prec e_B) \prec e_C = e_A \prec ( e_B \shuffle e_C), 
\\
(e_A \succ e_B) \prec e_C = e_A \succ ( e_B \prec e_C), 
\\
(e_A \shuffle e_B) \succ e_C = e_A \succ ( e_B \succ e_C). 
\end{align*}
\end{definition}

Using the additional identity $e_A \prec e_B = e_B \succ e_A$, equivalent to commutivity of $\shuffle$, we observe that the Left and Right Half-shuffles satisfy a Left and Right Zinbiel identity. Thus $\succ$ and $\prec$ are sometimes referred to as Paraproducts. For any geometric rough path and any two words $A$ and $B$ we have
\begin{equation}
\label{eq:ZinbielProperty}
\int_s^t \langle \rx_{s, r}, e_A\rangle d \langle \rx_{s, r}, e_B\rangle = \langle \rx_{s, t}, e_{A} \succ e_{B}\rangle = \langle \delta_\succ [ \rx_{s, t}], e_A \otimes e_B \rangle. 
\end{equation}
where $\delta_\succ$ is the Right Half-Unshuffle. Using the Right Half-Unshuffle, we are able to ``stitch'' two controlled rough paths together to obtain an object that will satisfy the Sewing Lemma, providing us with a meaningful way to integrate a Controlled Rough Path with respect to another controlled rough path. 

\begin{theorem}
\label{thm:LiftofControlledPath}
Let $\bY$ and $\bZ$ be $\rx$-controlled rough paths. Then by exploiting Equation \eqref{eq:ZinbielProperty} we obtain
\begin{align*}
&\int_0^T Y_t \otimes dZ_t 
\\
&= \lim_{|D|\to 0} \sum_{i: t_i\in D} Y_{t_i}\otimes Z_{t_i, t_{i+1}} + \Big\langle \Big(\bY_{t_i}-Y_{t_i}\Big) \otimes \Big( \bZ_{t_i} - Z_{t_i} \Big) , \delta_\succ[\rx_{t_i, t_{i+1}}] \Big\rangle. 
\end{align*}
In a similar fashion, we obtain
\begin{align*}
&\int_s^t Y_{s,r} \otimes dZ_r 
\\
&= \lim_{|D|\to 0} \sum_{i: t_i\in D} \Big\langle \Big( \bY_{t_i} - Y_{t_i} \Big) \otimes \Big( \bZ_{t_i} - Z_{t_i} \Big) , \delta_\succ[\rx_{t_i, t_{i+1}}] \Big\rangle. 
\end{align*}

Given an $\rx$-controlled rough path $\bY$, one can extend it to a rough path $\ry$ taking values in $G^M(U)$. Define the path $\ry:[0,T] \to G^M(U)$ by
\begin{align}
\label{eq:thm:LiftofControlledPath}
\ry_{s, t}=& \rId+ \sum_{k=1}^M \lim_{|D|\to 0} \sum_{i: t_i\in D} \Big\langle \big( \bY_{t_i} - Y_{t_i}\big)^{\otimes k}, (\delta_\succ)^{k}[\rx_{t_i, t_{i+1}}] \Big\rangle
\end{align}
where the iterated coproduct $(\delta_\succ)^k: T^M(V^*) \to T^M(V^*)^{\otimes k}$ is defined inductively by
$$
(\delta_\succ)^2 = ( (\delta_\succ) \otimes I) \delta_\succ, \quad (\delta_\succ)^{k+1} = ((\delta_\succ) \otimes I^{\otimes k}) \delta_\succ^k. 
$$
\end{theorem}

\begin{proof}[Proof of Theorem \ref{thm:LiftofControlledPath}]
The ideas behind this proof are well understood, although to the best of the authors knowledge have not been written using the language of Zinbiel algebras before. 

Firstly, 
$$
\int_s^t Y_r \otimes dZ_r = Y_s \otimes Z_{s, t} + \int_s^t Y_{s, r} \otimes dZ_r
$$
and from the definition of controlled rough paths we have
\begin{align*}
Y_s \otimes Z_{s, t} =& Y_s \otimes \langle \bZ_s, \rx_{s, t}-\rId\rangle + Y_s \otimes \langle R^Z_{s, t}, e_\varepsilon\rangle, 
\\
Y_{s, r} =& \langle \bY_{s}, \rx_{s, r} -\rId\rangle + \langle R^Y_{s, r}, e_\varepsilon\rangle. 
\end{align*}
Thus
\begin{align*}
&\int_s^t Y_{s, r} \otimes dZ_{r} 
\\
&= \Bigg( \Big\langle \bY_s - Y_s, \int_s^t \rx_{s, r} \Big\rangle + \int_s^t \langle R^Y_{s, r}, e_\varepsilon\rangle \Bigg) \otimes \Bigg( \Big\langle \bZ_s - Z_s, d\rx_{s, r} \Big\rangle + \langle dR_{s, r}^Z, e_\varepsilon\rangle \Bigg), 
\\
&=\Bigg\langle \Big( \bY_s-Y_s \Big) \otimes \Big( \bZ_s - Z_s\Big), \int_s^t \rx_{s, r} d\rx_{s, r} \Bigg\rangle + o\Big( |t-s| \Big)
\end{align*}
as $|t-s| \to 0$ where we use the identity from Equation \eqref{eq:ZinbielProperty} and the regularity of Definition \ref{dfn:ControlledRoughPath}. Similarly
$$
\int_s^t Y_r \otimes dZ_r = Y_s \otimes Z_{s, t} + \Big\langle ( \bY_s - Y_s) \otimes ( \bZ_s - Z_s), \delta_\succ[\rx_{s, t}] \Big\rangle + o(|t-s|). 
$$
Motivated by this, we verify the conditions of the Sewing Lemma (see \cite{frizhairer2014}*{Lemma 4.2}) with
$$
\Xi_{s, t}:=Y_s \otimes Z_{s, t} + \Big\langle ( \bY_s- Y_s) \otimes (\bZ_s-Z_s), \delta_\succ[ \rx_{s, t}] \Big\rangle. 
$$
Thus for $s<t<u \in [0,T]$, 
\begin{align*}
\delta \Xi_{s,t,u} =& \Xi_{s, u} - \Xi_{s, t} - \Xi_{t, u} 
\\
=&-Y_{s, t} \otimes Z_{t, u} + \Bigg\langle \Big( (\bY - Y) \otimes (\bZ - Z)\Big)_{s, t}, \delta_{\succ}[ \rx_{t, u}] \Bigg\rangle
\\
&+ \sum_{A, B} \Big((\bY_s-Y_s) \otimes (\bZ_s-Z_s)\Big) [e_A \otimes e_B] \Big\langle \rx_{s, t} \otimes \rx_{t, u}, \overline{\Delta}[ e_A \succ e_B] \Big\rangle
\end{align*}
where $\overline{\Delta}$ is the reduced Coproduct. Next, we substitute in for the increments using the identities
\begin{align*}
\langle \bY_t, e_A\rangle - \langle \bY_s, \rx_{s, t} \boxtimes e_A\rangle = \langle R_{s, t}, e_A\rangle, 
\\
\langle \bZ_t, e_B\rangle - \langle \bZ_s, \rx_{s, t} \boxtimes e_B\rangle = \langle R_{s, t}, e_B\rangle. 
\end{align*}
Next, we use Sweedler notation to represent the identity
$$
\overline{\Delta} [e_A \succ e_B] = \sum_{A', A''} \sum_{B', B''} = e_{A' \shuffle B'} \otimes e_{ A'' \succ B''}. 
$$

Therefore
$$
\sup_{s, t, u\in [0,T]} \frac{\delta \Xi_{s, t,u}}{|u-s|^{M\alpha}} =o\Big(|u-s|\Big).  
$$

The ideas behind this proof are well understood (see \cite{lyons2007differential}*{p.74}) where Y is the solution to a linear rough differential equation, although to the best of the authors' knowledge they have not been written before using the language of Zinbiel algebras and for general controlled rough paths. We refer the reader to the forthcoming preprint \cite{cass2020Instrinsic}, where a proof is given of this result.

\end{proof}


%


\section{Appendix}
\label{sec:AppendixA}

\begin{proof}[Proof of Lemma \ref{lem:DivisionQuantization+Truncation}]
From Definition \ref{dfn:OptimalQuantizer}, we have
\begin{align*}
\fE_{n,r}(\cL) =& \min_{h_1, ..., h_n\in E} \Bigg( \int_E \min_{i=1, ..., n} \big\| x - h_i\big\|_E^r d\cL(x) \Bigg)^{1/r}
\leq \min_{h_1, ..., h_n\in P_U[E] } \Bigg( \int_E \min_{i=1, ..., n} \big\| x - h_i\big\|_E^r d\cL(x) \Bigg)^{1/r}, 
\\
\leq& 2^{(r-1)/r} \min_{h_1, ..., h_n\in P_U[E] } \Bigg( \int \int_{P_U[E] \times (I-P_U)[E]} \min_{i=1, ..., n} \big\| P_U[x] - h_i\big\|_E^r d\cL\Big(P_U[x]\Big) d\cL\Big((I-P_U)[x]\Big) 
\\
&+ \int_E \big\| (I-P_U)[x] \big\|_E^r d\cL(x) \Bigg)^{1/r}, 
\end{align*}
since by the assumption that $P_U$ is a projection on $\cH$ (rather than $E$), the two laws $\cL\circ (P_U)^{-1}$ and $\cL\circ (I-P_U)^{-1}$ are independent with respect to the joint law $\cL$. Exploiting this, we get
\begin{align*}
\fE_{n,r}(\cL) \leq& 2^{(r-1)/r} \Bigg( \int_{(I-P_U)[E]} \fE_{n,r}(\cL\circ (P_U)^{-1})^r d\cL\Big( (I-P_U)[x]\Big) + \int_{E} \big\| x\big\|_E^r d\cL\Big( (I-P_U)[x]\Big) \Bigg)^{1/r}
\\
\leq& 2^{(r-1)/r} \Bigg( \fE_{n,r}(\cL\circ (P_U)^{-1} ) + \Big( \int_{E} \big\| x - P_U[x]\big\|_E^r d\cL(x) \Big)^{1/r} \Bigg). 
\end{align*}
\end{proof}

\begin{proof}[Proof of Proposition \ref{pro:OptimalKarhunen}]
Define the Covariance Kernel 
$ \cS:C^{\alpha, 0}([0,T]; \bR^{d'})^* \to C^{\alpha, 0}([0,T]; \bR^{d'})$ by 
$$
\cS[ f]_t = \bE\Big[ f(W) \cdot W_t\Big], \quad \cS = i i^*
$$
where $i i^*$ is the Spectral representation of $\cS$. For a Hilbert space $\cH$ and a Banach space $E$, we define the Operator $l$-topology, on the collection of bounded linear operators $i:\cH \to E$ to be
$$
l(i):=\bE\Bigg[ \Big\| \sum_{k\in \bN} i[h_k] \xi_k \Big\|_E^2 \Bigg]^{1/2}
$$
where $(h_k)_{k\in \bN}$ is an orthonormal basis of $\cH$ and $\xi_k$ are i.i.d normal random variables. It is well known, see for example \cite{figiel1979projections}, that the closure in the $l$-topology of the finite rank operators is the compact operators. We wish to find the finite dimensional operator that best approximates the Spectral representation $i$ of the Covariance Kernel $\cS$ of Brownian motion in the $l$-topology. 

We follow the methods of \cite{bay2017karhunen}. Using Theorem \ref{thm:Ciecielski}, we can equivalently think of $\cL^W$ as a law over the Banach spaces of sequences $(W_{pm})_{(p, m)\in \Delta}$ that satisfy
$$
\sup_{(p, m)\in\Delta} 2^{p(\alpha - 1/2)} |W_{pm}| < \infty, \quad \lim_{p\to \infty} 2^{p(\alpha-1/2)} \sup_{m=1, ..., 2^p} |W_{pm}| = 0. 
$$
Equivalently, we think of elements of the dual space $C^{\alpha, 0}([0,T]; \bR^{d'})^*$ as being sequences over $\Delta$ that satisfy
$$
f = (f_{pm})_{(p, m)\in\Delta}, \quad \| f\|_{\alpha-\textrm{H\"ol}, *} = \sum_{(p, m)\in \Delta} 2^{p(1/2-\alpha)} |f_{pm}|<\infty.
$$
It is well known that we work with the operator
$$
i(f) = \sum_{(p, m)\in \Delta} f_{pm} W_{pm}, 
$$
where $W_{pm} = \int_0^T W_{pm}(s) H_{pm}(s) ds$ are independent normally distributed random variables with mean 0 and variance 1. 

Thus 
$$
\cS\Big[(f_{pm})_{(p, m)\in \Delta}\Big](t) = \sum_{(p, m)\in\Delta} f_{pm} G_{pm}(t), \quad f\Big( \cS[f]\Big) = \sum_{(p, m)\in \Delta} |f_{pm}|^2 . 
$$

We wish to maximise this Quadratic form subject to the linear condition
$$
\| f\| = \sum_{(p, m)\in \Lambda} 2^{p(1/2-\alpha)} |f_{pm}|= 1.
$$
By a simple convexity argument, the functionals that attain this maximisation problem will be wavelet evaluation functionals and hence 
$$
\lambda^{(001)} = \sup_{\| f\|_{\alpha-\textrm{H\"ol},*} = 1} f\Big( \cS[f]\Big) = 1 = f^{(001)}\Big( \cS[f^{(001)}]\Big), 
$$
where $f^{(001)} = (f^{(001)}_{pm})_{(p,m)\in \Delta}$ satisfies $f^{(001)}_{00} = e_1$ and $f^{(001)}_{pm}=0$ else. We label $\cS[f^{(001)}](t) = x^{(001)}(t) = G_{00}(t) e_1 \in C^{\alpha, 0}([0,T]; \bR^{d'})$. We define $\cS_{\mu_{001}}[f]:=\lambda^{(001)} f(x^{(001)}) x^{(001)}$ and $\cS_{001}[f] = \cS[f] - \cS_{\mu_{001}}[f]$. 

By construction, we have that the operator $\cS_{\mu_{001}}$ is the Covariance Kernel of the 1-dimensional Gaussian measure that best approximates $\cL^W$ in mean square. Equivalently, the Spectral representation $\cS_{\mu_{001}}=i_{\mu_{001}}i_{\mu_{001}}^*$ yields that $i_{\mu_{001}}$ is the 1-dimensional operator that best approximates $i$ in the $l$-topology. 

By repeating this method, we obtain a sequence of so-called ``Rayleigh coefficients'' and ``Rayleigh Functionals'' parametrised by $(q, n, i)\in \Delta \times \{1, ..., d'\}$ as
$$
\lambda^{(qni)} = 2^{q(2\alpha - 1)}, \quad f^{(qni)} = (f^{(qni)}_{pm})_{(p, m)\in \Delta}, \quad f^{(qni)}_{pm} = \delta_{p, q} \delta_{m, n} e_i, 
$$
and elements $G_{qn}$ that are orthonormal in $\cH$. 

For fixed $N\in \bN$, we obtain the first $d'\cdot 2^N$ elements of these sequences. We construct the projection operator $P_N:C^{\alpha, 0}({[0,T]}, \bR^{d'}) \to C^{\alpha, 0}({[0,T]}, \bR^{d'})$ defined by
$$
P_N[x](t) = \sum_{(q, n)\in \Delta_N}\sum_{i=1}^{d'} f^{(qni)}(x) G_{qn}(t) e_i. 
$$

Next, we decompose the law $\cL^W = \mu_{N} * \cL_N^W$ where $\mu_N = \cL^W\circ P_N^{-1}$ and $\cL^W_N = \cL^W\circ (I-P_N)^{-1}$. $\mu_N$ is a $2^N$-dimensional multivariate Gaussian distribution. $\cL_N^W$ is a Gaussian measure over $C^{\alpha, 0}({[0,T]}, \bR)$ with Kernel $S_N$ that satisfies 
$$
\sup_{f\in C^{\alpha, 0}({[0,T]}, \bR)^*} f\Big( \cS_N[f]\Big)\leq 2^{(N+1)(\alpha - 1/2)}. 
$$

In particular, for a random variable $W$ with law $\cL^W$ we have that  random variable
$$
P_N[W](t) = \sum_{(p, m)\in \Delta_N} W_{pm} G_{pm}(t)
$$
has law $\mu_N$ and
$$
\sup_{f\in C^{\alpha, 0}({[0,T]}, \bR)^*} \bE\Big[ f\Big( W - P_N[W]\Big)^2 \Big] = 2^{(N+1)(2\alpha - 1)}. 
$$
\end{proof}

\begin{proof}[Proof of Theorem \ref{thm:ExtensionTheorem1}]
Let $V=\bR^{d'\times n}$ and $U=\bR^{d\times n}$ with alphabets $\cA$ and $\hat{\cA}$ both with $n$ subalphabets 
\begin{align*}
\cA^j=\Big\{ (i, j): i\in \{1, ..., d'\} \Big\},
\quad
\hat{\cA}^j=&\Big\{ (i, j): i\in \{1, ..., d\} \Big\}. 
\end{align*}
where $j\in \{1, ..., n\}$. Thus all the vector spaces $V^j$ are isomorphic to $\bR^{d'}$ and $U^j$ are isomorphic to $\bR^d$ but each $V^j$ and $U^j$ is distinct and identifyable. As with the normal subgroup constructed in Equation \eqref{eq:NormalSubGroup}, we know the normal subgroup that generates the cosets for the quotient group is
\begin{align*}
I^M\Big( \bR^{d'\times n}\Big) =& \Big\{ h\in T^M\Big(\bR^{d'\times n}\Big): \langle h, e_I\rangle =0, \forall I \mbox{ s.t. } \exists j\in\{1, ..., n\} \mbox{ with } I\in \cA^j \Big\}, 
\\
K^M\Big( \bR^{d'\times n}\Big) =& \exp_\boxtimes\Big( I^M\big( \bR^{d'\times n}\big) \Big). 
\end{align*}
For $s, t\in[0,T]$, let 
$$
\pi_{G^M\big( \bR^{d'\times n}\big) / K^M\big( \bR^{d'\times n} \big)}\Big[ \tilde{\rw}_{s, t} \Big] = \tilde{\rw}_{s, t} \boxtimes K^M\big( \bR^{d'\times n} \big) = \rw_{s,t}. 
$$
By Theorem \ref{thm:LiftofControlledPath}, we know this is equal to
\begin{equation}
\label{eq:thm:ExtensionTheorem1.1}
\rx_{s, t}:= \rId + \lim_{|D|\to 0} \sum_{i: t_i\in D} \Bigg( \sum_{k=1}^M \Big\langle \big( \bX_{t_i} - X_{t_i}\big)^{\otimes k}, (\delta_\succ)^k[ \rw_{t_i, t_{i+1}}] \Big\rangle + B(X_{t_i}) (t_{i+1}-t_{i}) \Bigg),
\end{equation}
and $\bX$ is defined as in Equation \eqref{eq:ControlledXFormula}. It is important to realise that the drift term, the only term that contains the ``measure like'' contributions, is only included in the first level of the signature. Measure dependencies are generally smoother than path dependencies and their higher regularity means they are $o\big( |D|^{1+}\big)$. 

Next for $t_i\in D$, we have
\begin{align*}
(\bX_{t_i} - X_{t_i})^{\otimes k} =& \Big( \Sigma(X_{t_i}) + \Sigma \star \Sigma(X_{t_i}) + ... + \Sigma^{\star (M-1)} (X_{t_i}) \Big)^{\otimes k}
\\
=&\sum_{\substack{l_1, ..., l_k=1\\ l_1+...+l_k\leq M}}^{M-1} \bigotimes_{m=1}^k \Sigma^{\star l_m}(X_{t_i}). 
\end{align*}

Using Definition \ref{dfn:BandSigma} and Lemma \ref{lem:FstarG} we have that there exists $f_j: U^j \to L\Big( (V^j)^{\oplus l_m}, U^j\Big)$ such that
$$
\Sigma^{\star l_m}(X_s) = \diag_{j=1, ..., n}\Big( f_j(\langle X_s, e_{(\cdot, j)} ) \Big). 
$$
Similarly, there exist functions $g_j:U^j \to L\Big( (V^j)^{\oplus (l_1+...+l_k)}, U^j\Big)$ such that
$$
\bigotimes_{m=1}^k \Sigma^{\star l_m}(X_s) = \diag_{j=1, ..., n}\Big( g_j(\langle X_s, e_{(\cdot, j)} ) \Big), 
$$
which is an operator restricted to the subgroup $\oplus{j=1}^n G^M(V^j)$. Thus Equation \eqref{eq:thm:ExtensionTheorem1.1} is dependent on the tensor of rough paths $\rw$ and not on the Extension $\tilde{\rw}$. 
\end{proof}


\bibliographystyle{plain}


\end{document}